%% file: main.tex
\documentclass[a4paper,11pt]{scrartcl}
\pdfoutput=1

\input{header-packages.tex}
\input{header-environments.tex}
\input{header-macros.tex}

\bibliography{bibliography.bib}

\title{Vertical configuration spaces\\and their homology}
\author{Andrea Bianchi and Florian Kranhold}

\begin{document}

\maketitle

\begin{abstract}
  We introduce ordered and unordered configuration spaces of ‘clusters’ of
  points in an Euclidean space $\R^d$, where points in each cluster satisfy
  a ‘verticality’ condition, depending on a decomposition $d=p+q$. We compute
  the homology in the ordered case and prove homological stability in the
  unordered case.\par{}\vspace*{10px}
  \footnotesize \noindent\textbf{Date.} 22\textsuperscript{nd} March, 2021. \textbf{Last change.} \today.\\
  \textbf{Key words.} Configuration spaces, homological stability, clusters.\\
  \textbf{2020 \acr{MSC}.}
  Primary \texttt{55R80};      %
  secondary \texttt{55R25},    %
  \texttt{55R20},              %
  \texttt{55M99}.\vspace*{3px} %
\end{abstract}

\section{Introduction}
\input{01-introduction.tex}

\section{Preliminaries}
\label{sec:preliminaries}
\input{02-preliminaries.tex}

\section{The cohomology of \texorpdfstring{$\tilde{V}_K(\R^{p,q})$}{VK(Rpq)}}
\label{sec:orderedHomology}
\input{03-ray_partitions.tex}

\section{Homological stability}
\label{sec:stability}
\input{04-homological_stability.tex}

\printbibliography[heading=bibintoc]

\begin{addr}
  \auth
  {6cm}
  {Andrea Bianchi}
  {Department of Mathematical Sciences,\\
    University of Copenhagen,\\
    Universitetsparken 5,\\
    DK-2100 Copenhagen, Denmark}
  {anbi@math.ku.dk}
  \hspace*{8mm}
  \auth
  {5.2cm}
  {Florian Kranhold}
  {Mathematical Institute,\\
    University of Bonn,\\
    Endenicher Allee 60,\\
    53115 Bonn, Germany}
  {kranhold@math.uni-bonn.de}
\end{addr}

\end{document}

%% file: header-packages.tex
\usepackage[utf8]{inputenc}
\usepackage[T1]{fontenc}
\addtokomafont{disposition}{\rmfamily}              %
\usepackage[british]{babel}
\usepackage{csquotes}                               %
\usepackage[UKenglish,nodayofweek]{datetime}
\usepackage{a4}
\date{}

\usepackage[pdfusetitle,bookmarksnumbered=true]{hyperref}
\usepackage{xcolor}
\hypersetup{
  colorlinks,
  linkcolor={red!45!black},
  citecolor={green!45!black},
  urlcolor={blue!45!black}
}

\usepackage{lmodern}                                %
\usepackage{amssymb,amsmath,mathtools}              %
\usepackage[osf,sc]{mathpazo}                       %
\usepackage{bm}                                     %
\usepackage{ellipsis}                               %
\usepackage[babel]{microtype}                       %
\usepackage{amsthm}                                 %
\usepackage[scr=boondoxo,scrscaled=.98]{mathalfa}   %
\usepackage{relsize}                                %

\usepackage{tikz-cd}
\usetikzlibrary{quotes,babel,angles,arrows}
\usepackage{tikz}
\usepgflibrary{arrows}
\usetikzlibrary{shapes.arrows,trees,backgrounds,decorations.markings,shadows}

\usepackage[backref,%
            style=alphabetic,%
            backend      = biber,%
            maxbibnames  = 9,%
            maxcitenames = 9,%
            bibencoding  = utf8]{biblatex}
\DeclareFieldFormat[article]{volume}{\textbf{#1}}       %
\DeclareFieldFormat[misc]{title}{‘#1’}                  %
\DeclareFieldFormat[book]{number}{\textbf{#1}}          %
\DeclareFieldFormat[inproceedings]{number}{\textbf{#1}} %
\AtBeginBibliography{\small}
\DefineBibliographyStrings{english}{
  backrefpage={$\uparrow$},
  backrefpages={$\uparrow$}
}

\usepackage{graphicx}
\usepackage{overpic}
\usepackage[margin=10pt,font=small,labelfont=bf,labelsep=period,format=plain]{caption}

%% file: header-environments.tex
\theoremstyle{plain}
\newtheorem{theo} {Theorem}[section]
\newtheorem{lem}  [theo]{Lemma}
\newtheorem{prop} [theo]{Proposition}

\newtheorem{conj} [theo]{Conjecture}

\theoremstyle{definition}
\newtheorem{defi}  [theo]{Definition}

\newtheorem{rem}   [theo]{Remark}
\newtheorem{expl}  [theo]{Example}
\newtheorem{constr}[theo]{Construction}
\newtheorem{nota}  [theo]{Notation}
\newtheorem{outl}  [theo]{Outlook}

\newenvironment{addr}{ %
~\\[-20px]
\begingroup
\small
\begin{flushright}
}{
\end{flushright}
\endgroup
}

\newcommand{\auth}[4]{ %
\begin{minipage}[t]{#1}
\small
{\textbf{#2}}\\[3px]
#3\\[3px]
\texttt{\href{mailto:#4}{#4}}
\end{minipage}
}

\makeatletter
\def\blfootnote{\gdef\@thefnmark{}\@footnotetext}

\makeatother

%% file: header-macros.tex
\let\setminus\smallsetminus
\let\le\leqslant
\let\ge\geqslant
\let\leq\leqslant
\let\geq\geqslant
\let\emptyset\varnothing
\let\phi\varphi
\let\epsilon\varepsilon
\let\theta\vartheta

\newcommand{\ol}     [1]{\overline{#1}}

\newcommand{\on}     [1]{\text{\textup{#1}}}

\newcommand{\quot}   [2]{\left.\raise0.5ex\hbox{$#1$} \right/\hspace*{-2px}\lower0.5ex\hbox{$#2$}}
\newcommand{\llangle}[1]{\mathopen{}\left\langle #1 \right\rangle\mathclose{}}
\newcommand{\pa}[1]{\mathopen{}\left(#1\right)\mathclose{}}     %
\newcommand{\abs}[1]{\lvert #1\rvert}                           %
\newcommand{\set}[1]{\mathopen{}\left\{#1\right\}\mathclose{}}  %
\newcommand{\Th}[1]{{\bfseries\small #1}}                       %
\newcommand{\acr} [1]{\texorpdfstring{\textsmaller{#1}}{#1}}    %

\newcommand{\N}{\mathbb{N}}      %
\newcommand{\R}{\mathbb{R}}      %
\newcommand{\Z}{\mathbb{Z}}      %

\newcommand{\cpt}{^{\infty}}     %
\newcommand{\V}{\mathscr{V}}     %
\newcommand{\fF}{\mathfrak{F}}   %
\newcommand{\pr}{\on{pr}}        %

\newcommand{\Q}{\caQ}            %

\newcommand{\bbP}{\mathbb{P}}

\newcommand{\bE}{\mathbb{E}}

\newcommand{\caO}{\mathcal{O}}
\newcommand{\caP}{\mathcal{P}}
\newcommand{\caQ}{\mathcal{Q}}

\newcommand{\scC}{\mathscr{C}}

\ifdefined\bfq\renewcommand{\bfq}{\mathbf{q}}\else \newcommand{\bfq}{\mathbf{q}}\fi

\newcommand{\frF}{\mathfrak{F}}

\newcommand{\frS}{\mathfrak{S}}

\newcommand{\tV}{\tilde{V}}     %
\newcommand{\tH}{\tilde{H}}     %

\definecolor{dgreen}{rgb}{.25,.7,.3}
\definecolor{bgreen}{rgb}{.5,.9,.5}
\definecolor{dyellow}{rgb}{.85,.7,0}
\definecolor{byellow}{rgb}{.95,.93,.4}
\definecolor{dred}{rgb}{.9,.3,.3}
\definecolor{dblue}{rgb}{.3,.3,.8}
\definecolor{bred}{rgb}{.95,.6,.6}
\definecolor{bblue}{rgb}{.5,.5,.9}
\definecolor{dgrey}{rgb}{.6,.6,.6}
\definecolor{lila}{rgb}{.8,.6,.9}
\definecolor{blila}{rgb}{.94,.84,.97}
\definecolor{bgrey}{rgb}{.9,.9,.9}
\definecolor{ddgreen}{rgb}{0,.5,0}
\definecolor{grey}{rgb}{.87,.87,.87}

\newcommand{\stab}{\mathrm{stab}}
\newcommand{\sg}{\on{sg}}
\newcommand{\id}{\mathrm{id}}
\newcommand{\Id}{{\mathrm{Id}}}

%% file: 01-introduction.tex
We fix integers $p\ge 0$ and $q\ge 1$ and let $d\coloneqq p+q$ throughout the article.
For $k\geq 1$, a \emph{cluster of size $k$} in $\R^d$ is a tuple of $k$ distinct
points of $\R^d$, i.e.\ a point in the ordered configuration space $\tilde{C}_k(\R^d)$
of $k$ points in $\R^d$. For $r\geq 0$ and a tuple $K=(k_1,\dots,k_r)$ of integers
$k_i\geq 1$, we consider the subspace
\[\tilde{C}_K(\R^d)\subseteq \prod_{i=1}^r (\R^d)^{k_i}\]
consisting of all configurations of $r$ ordered, pairwise disjoint clusters of
sizes $k_1,\dots,k_r$ in $\R^d$. Note that if we put
$\abs{K}\coloneqq k_1+\dotsb+k_r$, the space $\tilde{C}_K(\R^d)$ is, up to
reindexing, homeomorphic to the more familiar space
$\tilde{C}_{\abs{K}}(\R^d)$.

Now decompose $\R^d=\R^p\times \R^q$, and denote by $\pr_1\colon\R^d\to\R^p$ the
projection on the first $p$ coordinates. A cluster $z=(z^1,\dots,z^k)$ of $k$
points in $\R^d$ is \emph{vertical} if $\pr_1(z^1)=\dotsb=\pr_1(z^k)$, i.e.\ the
$k$ points in the cluster share their first $p$ coordinates. For $p=1$ and
$q=1$, we are requiring the $k$ points of the cluster to lie on the same
vertical line of $\R^2$, whence the terminology: see Figure \ref{fig:vtilde}.

\begin{defi}
  \label{defi:bV}
  For $r$ and $K=(k_1,\dots,k_r)$ as above, we introduce a subspace
  \[\tV_K(\R^{p,q})\subseteq \tilde{C}_K(\R^{p+q}).\]
  A sequence of clusters $(z_1,\dotsc,z_r)$ belongs to $\tV_K(\R^{p,q})$ if
  and only if each cluster $z_i=(z_{i}^1,\dots,z_i^{k_i})$ is vertical.
\end{defi}

\begin{figure}
  \centering
  \input{fig01}
  \caption{A configuration in $\tV_{(\textcolor{red}{3},\textcolor{dgreen}{4},
      \textcolor{blue}{2},\textcolor{dyellow}{2})}(\R^{1,1})$}\label{fig:vtilde}
\end{figure}

\noindent These spaces have already been studied in \cite{latifi}. Our interest
for the spaces $\tV_K(\R^{p,q})$ has the following reasons.
\begin{enumerate}
\item[\Th{1.}] There is a coloured operad $\V_{p,q}$, in spirit
  similar to the operad of little cubes and, even more closely, to the extended
  Swiss cheese operad \cite{Willwacher}: the second author has introduced this operad in his PhD thesis, and has studied the problem of delooping $\V_{p,q}$-algebras in \cite{kranhold2}. The spaces $\tV_K(\R^{p,q})$
  occur, up to a mild homotopy equivalent replacement, in the description of the
  operad $\V_{p,q}$. The first occurrence of operations related to the operad
  $\V_{p,q}$ is in \cite[§5]{boedigop}.
\item[\Th{2.}] For $n\ge0$, the cohomology of the ordered configuration
  space $\tilde{C}_n(\R^d)$ is known to be free abelian; more precisely, for
  every choice of
  \begin{itemize}
  \item $r\ge1$, and a sequence $K=(k_1,\dots,k_r)$ with $\abs{K}=n$;
  \item a partition of the set $\set{1,\dots,n}$ into pieces of sizes $k_1,\dots,k_r$,
  \end{itemize}
  we have a proper embedding
  $\tV_K(\R^{d-1,1})\hookrightarrow \tilde{C}_n(\R^d)$, and a basis for
  $H^*(\tilde{C}_n(\R^d);\Z)$ can be chosen to consist of Poincaré duals of
  components of the submanifolds $\tV_K(\R^{d-1,1})$ obtained in this way.
\item[\Th{3.}] The spaces $\tV_K(\R^{p,q})$ give an example of ordered
  configuration spaces for which the Fadell–Neuwirth maps fail, in
  general, to be fibrations.
\end{enumerate}
There is an unordered counterpart of the construction above: consider the
partition of $\set{1,\dots,\abs{K}}$ into $r$ consecutive segments of lengths
$k_1,\dots,k_r$ and denote by $\frS_K\subseteq \frS_{\abs{K}}$ the
subgroup of the symmetric group containing all permutations $\sigma$ which
preserve this partition, i.e.\ $\sigma$ maps each partition component to a
(possibly different) partition component. The group $\frS_K$ can be
described as follows: for all $k\ge1$ we denote by $r(k)\ge0$ the number of
occurrences of $k$ in $K$; then $\frS_K$ is isomorphic to the product
\[\frS_K\cong \prod_{k=1}^\infty\frS_k\wr\frS_{r(k)}
  = \prod_{k=1}^\infty(\frS_k)^{r(k)}\rtimes\frS_{r(k)}.\]

\begin{defi}\label{defi:unorderedbV}
  The group $\frS_K$ acts freely on $\tV_K(\R^{p,q})$ by permuting the labels
  $1\leq i\leq r$ of clusters of the same size, and permuting the labels
  $1\leq j\leq k_i$ of the points of each cluster; we denote the quotient space by
  \[V_K(\R^{p,q})\coloneqq \tV_K(\R^{p,q})/\frS_K.\]
\end{defi}

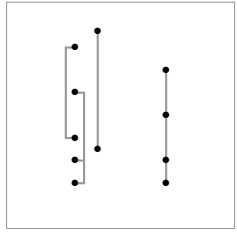
\begin{figure}
  \centering
  \input{fig02}
  \caption{A configuration in $V_{(3,4,2,2)}(\R^{1,1})$}\label{fig:vconf}
\end{figure}

Roughly speaking, and using the notation above, a point in $V_K(\R^{p,q})$
consists of a collection of $r$ clusters, of which $r(k)$ have size $k$;
clusters of the same size are unordered, and points inside a cluster are also
unordered, see Figure \ref{fig:vconf}. One can thus regard $V_K(\R^{p,q})$ as a subspace of
\[\prod_{k=1}^\infty \pa{C_k(\R^d)^{r(k)}/\frS_{r(k)}},\]
where $C_k(\R^d)$ denotes the unordered configuration space of $k$ points in
$\R^d$. Our interest for the spaces $V_K(\R^{p,q})$ has the following reasons.
\begin{enumerate}
\item[\Th{1.}] These spaces occur naturally in the description of
  free algebras over the operad $\V_{p,q}$ mentioned earlier.
\item[\Th{2.}] We note that for $p=0$ and $q=d$, the space
  $\tV_K(\R^{0,d})$ is homeomorphic to the ordered configuration space
  $\tilde{C}_{\abs{K}}(\R^d)$; however the unordered version
  $V_K(\R^{0,d})$ is in general not homeomorphic to $C_{\abs{K}}(\R^d)$,
  rather it is a covering of the latter space; the space $V_K(\R^{0,d})$ is
  an unordered configuration space of \emph{clusters of points} in $\R^d$,
  without any ‘verticality’ condition. For $d=2$, spaces of unordered
  configurations of clusters have been considered in their own sake in
  \cite{tran} and in relation to Hurwitz spaces in \cite{tietz}.
\item[\Th{3.}] For general $p,q$, the spaces $V_K(\R^{p,q})$ are
  related to the spaces of parallel submanifolds in an ambient manifold, see
  \cite{palmerdis} and \cite{latifi}.
\end{enumerate}
We will occasionally restrict to situations where all clusters have the same
size, i.e.\ $K=(k,\dotsc,k)$ for some $k\ge 1$ and let $r$ be the length of
this tuple. In these situations, we will we will simplify our notation and write
\[\tilde{V}^k_r(\R^{p,q}) \coloneqq\tV_{(k,\dotsc,k)}(\R^{p,q})
  \quad\text{and}\quad V^k_r(\R^{p,q})\coloneqq V_{(k,\dotsc,k)}(\R^{p,q}).\]

\paragraph{Results}
The first aim of this article is to compute the integral homology of the spaces
$\tV_K(\R^{p,q})$, for all choices of $p$, $q$, and $K$. We will see that
$H_*(\tV_K(\R^{p,q});\Z)$ is free abelian, and for $q=1$ it is supported in
degrees multiple of $p$.

The second aim is to prove a homological stability result: for all $k\geq 1$,
the stabilisation map $V^k_r(\R^{p,q})\to V^k_{r+1}(\R^{p,q})$, which adjoins a
new cluster of size $k$, induces isomorphisms in integral homology in degrees
$*\le\tfrac{r}{2}$.  This extends the results of \cite{tran}, \cite{palmerdis},
and \cite{latifi} who covered all cases with $p+q\geq 3$.

\paragraph{Outline}
Our article is organised as follows: in Section \ref{sec:preliminaries}, we
introduce some notation and make some first obversations about the basic
properties of these configuration spaces. In Section \ref{sec:orderedHomology},
we calculate the integral homology of the ordered vertical configuration spaces
$\tV_K(\R^{p,q})$. Then we turn in Section \ref{sec:stability} to the question
of homological stability for the unordered vertical configuration spaces
$V^k_r(\R^{p,q})$ with fixed cluster size $k$.

\paragraph{Funding}
This work was supported by the \emph{Deutsche Forschungsgemeinschaft} (German
Research Foundation) [\texttt{EXC-2047/1}, \texttt{390685813}, to A.B. and
F.K.]; the \emph{Danish National Research Foundation} through the \emph{Centre
  for Geometry and Topology} [\texttt{DNRF151}, to A.B.]; the \emph{European
  Research Council} under the \emph{European Union’s Seventh Framework Programme}
  [\texttt{ERC StG 716424 - CASe}, to A.B.]; the
\emph{Max Planck Institute for Mathematics} in Bonn [to F.K.]; and the \emph{Promotionsf\"orderung} of the \emph{Stu\-dien\-stif\-tung des Deutschen
  Volkes} [to F.K.].

\paragraph{Acknowledgements}
This project started in 2018, when both authors were PhD students of
Carl-Friedrich Böodigheimer. We would like to thank him for suggesting the
study of vertical configuration spaces and for numerous enlightening
con\-versations on the subject.

This paper also benefited from discussions with Martin Palmer about homological
stability of configuration spaces and about twisted coefficient systems, and
from many useful comments he made on a first draft.

We would like to thank Genta Latifi for sharing her master’s thesis with us.

We are grateful to Oscar Randal-Williams for a helpful conversation about
homological stability of unordered configuration spaces with labels.

Finally, we thank the anonymous referee for careful reading of the paper and
helpful comments.

%% file: fig01.tex
\begin{tikzpicture}[scale=3]
  \draw[thin,dgrey] (0,0) rectangle (1,1);
  \node[red] at (.3,.2) {\tiny $\bullet$};
  \node[red] at (.3,.3) {\tiny $\bullet$};
  \node[blue] at (.3,.4) {\tiny $\bullet$};
  \node[red] at (.3,.6) {\tiny $\bullet$};
  \node[blue] at (.3,.8) {\tiny $\bullet$};
  \node[red] at (.22,.2) {\tiny $1,2$};
  \node[red] at (.22,.3) {\tiny $1,1$};
  \node[blue] at (.22,.4) {\tiny $3,1$};
  \node[red] at (.22,.6) {\tiny $1,3$};
  \node[blue] at (.22,.8) {\tiny $3,2$};
  \node[dgreen] at (.7,.5) {\tiny $\bullet$};
  \node[dgreen] at (.7,.7) {\tiny $\bullet$};
  \node[dgreen] at (.7,.2) {\tiny $\bullet$};
  \node[dgreen] at (.7,.3) {\tiny $\bullet$};
  \node[dgreen] at (.78,.5) {\tiny $2,1$};
  \node[dgreen] at (.78,.7) {\tiny $2,3$};
  \node[dgreen] at (.78,.2) {\tiny $2,4$};
  \node[dgreen] at (.78,.3) {\tiny $2,2$};
  \node[dyellow] at (.4,.87) {\tiny $\bullet$};
  \node[dyellow] at (.4,.35) {\tiny $\bullet$};
  \node[dyellow] at (.48,.87) {\tiny $4,1$};
  \node[dyellow] at (.48,.35) {\tiny $4,2$};
\end{tikzpicture}

%% file: fig02.tex
\begin{tikzpicture}[scale=3]
  \draw[thin,dgrey] (0,0) rectangle (1,1);
  \draw[dgrey,thick] (.7,.7) -- (.7,.2);
  \draw[dgrey,thick] (.4,.87) -- (.4,.35);
  \draw[dgrey,thick] (.3,.8) -- (.26,.8) -- (.26,.4) -- (.3,.4);
  \draw[dgrey,thick] (.3,.6) -- (.34,.6) -- (.34,.2) -- (.3,.2);
  \draw[dgrey,thick] (.3,.3) -- (.34,.3);
  \node at (.3,.2) {\tiny $\bullet$};
  \node at (.3,.3) {\tiny $\bullet$};
  \node at (.3,.4) {\tiny $\bullet$};
  \node at (.3,.6) {\tiny $\bullet$};
  \node at (.3,.8) {\tiny $\bullet$};
  \node at (.7,.5) {\tiny $\bullet$};
  \node at (.7,.7) {\tiny $\bullet$};
  \node at (.7,.2) {\tiny $\bullet$};
  \node at (.7,.3) {\tiny $\bullet$};
  \node at (.4,.87) {\tiny $\bullet$};
  \node at (.4,.35) {\tiny $\bullet$};
\end{tikzpicture}

%% file: 02-preliminaries.tex
In this section we introduce notation for the spaces $\tV_K(\R^{p,q})$
and $V_K(\R^{p,q})$, and make some basic observations about the topology of
these spaces.

\begin{nota}
  \label{nota:xyt}
  Recall that $\pr_1\colon \R^d\to \R^p$ denotes the projection to the first $p$
  coordinates. However, in several situations, we will make use of the
  decomposition $\R^d=\R^{p+q-1}\times \R$ and write $(\zeta,t)$ for a generic
  point in $\R^d$. Hence, we have two other projections, namely
  $\pr_\zeta\colon\R^d\to \R^{p+q-1}$ and $\pr_t\colon \R^d\to \R$. Clearly, if
  $q=1$, then $\pr_\zeta$ and $\pr_1$ coincide.
\end{nota}

\begin{nota}
  \label{nota:suggestivesum}
  We denote elements in $\tV_K(\R^{p,q})$ resp.\ $V_K(\R^{p,q})$ as
  follows:
  \begin{itemize}
  \item An element in $\tV_K(\R^{p,q})$ is an ordered collection
    $Z\coloneqq(z_1,\dotsc,z_r)$ of (vertical) clusters
    $z_i\coloneqq (z_i^1,\dotsc,z_i^{k_i})$. We will also often write
    $Z=(z_1^1,\dotsc,z_r^{k_r})$.
  \item For the unordered version, we use the suggestive sum notation: A generic
    element in $V_K(\R^{p,q})$ is an unordered collection
    $[Z]\coloneqq\sum_{i=1}^r[z_i]$ of unordered (vertical) clusters
    $[z_i]\coloneqq [z_i^1,\dotsc,z_i^{k_i}]=\{z_i^1,\dotsc,z_i^{k_i}\}$.
  \end{itemize}
\end{nota}

\begin{rem}[Path components of \texorpdfstring{$\tV_K(\R^{p,q})$}{Ṽk(Rpq)}
  and \texorpdfstring{$V_K(\R^{p,q})$}{Vk(Rpq)}]
  \label{rem:pi0Vk}
  For the ordered vertical configuration spaces, the following hold:
  \begin{itemize}
  \item For $q\ge 2$ the space $\tV_K(\R^{p,q})$ is connected.
  \item For $q=1$ and $p\ge 1$, the space $\tV_K(\R^{p,q})$ has one component
    $\tV_K(\R^{p,q})_{\Sigma}$ for each tuple
    $\Sigma=(\sigma_1,\dotsc,\sigma_r)\in \prod_i\frS_{k_i}$ of
    permutations. This component contains all configurations
    $(z_{1}^1,\dotsc,z_r^{\smash{k_r}})$ with
    $\smash{\pr_t(z_i^{\smash{\sigma_i(j)}})<\pr_t(z_i^{\smash{\sigma_i(j+1)}})}$
    for all $1\le i\le r$ and $1\le j< k_i$.
  \item For $q=1$ and $p=0$, we note that $\tV_K(\R^{0,1})=\tilde{C}_{\abs{K}}(\R)$,
    so each permutation $\sigma\in\frS_{|K|}$ corresponds to a connected component
    which contains all configurations $(z_1^1,\dotsc,z_r^{\smash{k_r}})=
    (z^1,\dotsc,z^{\abs{K}})$ with $z^i<z^{\sigma(i)}$.
  \end{itemize}
  We have inclusions
  $\prod_i\frS_{k_i}\subseteq\frS_{K}\subseteq \frS_{|K|}$ and the group
  $\frS_{K}=\prod_k\frS_k\wr\frS_{r(k)}$ acts on $\pi_0\tV_K(\R^{p,q})$
  with quotient equal to $\pi_0V_K(\R^{p,q})$. Since the action is transitive
  in the first two cases listed above, the space $V_K(\R^{p,q})$ is connected
  for $(p,q)\ne (0,1)$, whereas for $(p,q)=(0,1)$ we can identify
  \[\pi_0V_K(\R^{0,1}) \cong \frS_{|K|}/\frS_K.\]
  The latter set can also be identified with the set of unordered partitions of
  $\{1,\dotsc,|K|\}$ into subsets of sizes $k_1,\dotsc,k_r$; such that for
  all $k\ge1$ there are $r(k)$ partition components of size $k$.
\end{rem}

\begin{rem}[\texorpdfstring{$\tV$}{Ṽ} and \texorpdfstring{$V$}{V} are manifolds]
  The space $\tV_K(\R^{p,q})$ is an open subspace of $(\R^p)^r\times
  (\R^q)^{|K|}$ and hence an orientable smooth manifold of dimension $p\cdot r
  + q\cdot |K|$. The action of $\frS_{K}$ is free, so $V_K(\R^{p,q})$ is
  again a manifold of the same dimension. The manifold $V_K(\R^{p,q})$ is
  non-orientable if and only if at least one of the following holds:
  \begin{itemize}
  \item $q\geq 3$ is odd and there is at least one cluster of some size $k\ge2$;
    then a path in $V_K(\R^{p,q})$ interchanging two points of this cluster,
    while fixing all other points, reverses the local orientation;
  \item $p+q\ge2$ and there is some $k\ge1$ such that $p+q\cdot k$ is odd and
    $r(k)\ge 2$; then, interchanging two clusters of size $k$ while preserving
    their internal ordering and fixing all other points reverses the local
    orientation.\vspace*{4px}
  \end{itemize}
\end{rem}

\begin{rem}[Poincaré–Lefschetz duality]\label{rem:PLD}
  For a topological space $X$, we denote by $X^\infty$ its one-point
  compactification, and denote the point at infinity by $\infty$. Since
  $\tV_K(\R^{p,q})$ is an open and orientable manifold of dimension
  $p\cdot r+q\cdot \abs{K}$, we can apply Poincaré–Lefschetz duality and
  obtain
  \[H^*\pa{\tV_K(\R^{p,q})}\cong
    H_{p\cdot r+q\cdot|K|-*}\pa{\tV_K(\R^{p,q})^\infty,\infty}.\]
\end{rem}

%% file: 03-ray_partitions.tex
In this section we calculate the integral cohomology of the spaces
$\tV_K(\R^{p,q})$ for all dimensions $p\ge 0$ and $q\ge 1$. In the case
$p=0$ we recover the calculations of \cite{arnold} and
\cite[§\,\textsc{iii}.6]{clm} for the classical
configuration spaces $\tilde{C}_{\abs{K}}(\R^d)$. Let us exclude the case
$(p,q)=(0,1)$, where all components are contractible.

\subsection{Ray partitions}

\noindent We fix a partition $K=(k_1,\dots,k_r)$ for the entire
section. Before we state our main result about the cohomology of
$\tV_K(\R^{p,q})$, we need to introduce a few combinatorial notions.

\begin{defi}
  \label{defi:indset}
  The \emph{table} associated with the partition $K$ is the set
  \[T_K\coloneqq \{(i,j);\, 1\le i\le r\text{ and }1\le j\le k_i\}.\]
  We order $T_K$ lexicographically, which means we write $(i,j)<(i',j')$
  if either $i<i'$, or $i=i'$ and $j<j'$ holds.
\end{defi}
\begin{nota}
  \label{nota:lengthagility}
  For each partition $\Q$ of $T_K$ into non-empty subsets $\Q_1,\dotsc,\Q_l$
  we consider two positive integers:
  \begin{itemize}
  \item The number $l(\Q)\coloneqq l$ is called the \emph{length}
    of the partition, and in general we have $1\le l(\Q)\le \abs{K}$.
  \item Consider on $\{1,\dotsc,l\}$ the equivalence relation spanned by
    $\beta\sim\beta'$ if there are $1\le i\le r$ and $1\le j,j'\le k_i$ with
    $(i,j)\in \Q_\beta$ and $(i,j')\in \Q_{\beta'}$ (i.e.\ the
    $i$\textsuperscript{th} cluster intersects both $\Q_\beta$ and
    $\Q_{\beta'}$). The number of equivalence classes
    $1\le a(\Q)\le \min(l(\Q),r)$ will be called the \emph{agility} of the
    partition.
  \end{itemize}
\end{nota}

\begin{defi}\label{defi:raypartition}
  A \emph{ray partition $\Q$ of type $K$} is a partition $\Q_1,\dotsc,\Q_l$
  of $T_K$, with a total order $\prec_{\beta}$ on each piece $\Q_{\beta}$
  (called \emph{ray}), such that the following hold:
  \begin{enumerate}
  \item [\Th{R1.}] the components are labelled from $1$ to $l$
    according to their minimum with respect to the global order $<$, i.e.
    \[\min(\Q_1,<)<\dotsb<\min(\Q_l,<);\]
  \item [\Th{R2.}] for each $1\le\beta\le l$, the minima
    with respect to $<$ and $\prec_\beta$ coincide
    \[\min(\Q_\beta,\prec_\beta) = \min(\Q_\beta,<).\]
  \end{enumerate}
\end{defi}

\begin{defi}
  \label{defi:witnessed}
  Let $Z=(z_{1}^1,\dotsc,z_r^{\smash{k_r}})\in \tV_K(\R^{p,q})$. 
  We say that a ray partition $\Q$ is \emph{witnessed by $Z$} if the
  following conditions hold:
  \begin{enumerate}
  \item[{\small \bfseries W1.}] all $z_{i}^{\smash j}$ with $(i,j)\in \Q_\beta$
    project along $\pr_{\zeta}$ to the same point in $\R^{d-1}$.\vspace*{-5px}
  \item[{\small \bfseries W2.}] if $(i,j)\prec_{\beta}(i',j')$ in $\Q_{\beta}$,
    then $\pr_t(z_{i}^{\smash j})<\pr_t(z_{i'}^{\smash{j'}})$ in $\R$.
  \end{enumerate}
\end{defi}
Condition {\small\textbf{W1}} says that the points $z_{i,j}$ with
$(i,j)\in \Q_\beta$ lie on a line in $\R^d$ parallel to the $t$-axis;
condition {\small\textbf{W2}} ensures that the same points are assembled
on this line according to the order $\prec_{\beta}$ of their indices.
In particular, the points $z_{i,j}$ with $(i,j)\in \Q_\beta$ lie on a \emph{ray},
namely the half-line starting at $z_{\min(\Q_\beta,\prec_\beta)}$ and running
in the positive $t$-direction. See Figure \ref{fig:ray} for an example.

\begin{figure}
  \centering
  \input{fig03}
  \caption{A configuration in the space
    $\tV_{(\textcolor{red}{3},\textcolor{dgreen}{4},\textcolor{blue}{3})}(\R^{1,2})$
    which witnesses the ray par\-ti\-tion $(\Q_1,\dotsc,\Q_7)$, where e.g.\
    $\Q_2=\{(1,2)\prec (3,3)\prec (1,3)\}$. The components $\Q_\beta$ are numbered
    according to their smallest label ({\small\textbf{R1}}), and the point carrying
    the minimal label lies at the bottom of each ray ({\small\textbf{R2}}). Recall
    that the verticality condition demands that all points belonging to the same
    cluster have to lie in the same affine plane orthogonal to the first axis. Note that this ray partition seems
    not to be the most ‘efficient’ one: we may merge $\Q_4$ and $\Q_5$. We will
    introduce a measure for ‘efficiency’ soon.}\label{fig:ray}
\end{figure}

\begin{rem}\label{rem:q1}
  If $q=1$, recall that we consider only $p\ge 1$. Then
  $\tV_K(\R^{p,1})$ is disconnected, with path components indexed by
  tuples $\Sigma\in\prod_{i=1}^r\frS_{k_i}$ (see Remark \ref{rem:pi0Vk}), and we
  would like to calculate the homology of a single path component. In order to
  do so, we assign to each ray partition $\Q$ of type $K$ such a tuple
  $\Sigma$ as follows.
\end{rem}

\begin{defi}
  \label{defi:stackedprec}
  Given a ray partition $\Q$, the ‘stacked’ total order $\prec$ on
  \[T_K=(\Q_l,\prec_l)\sqcup \dotsb\sqcup (\Q_1,\prec_1),\]
  is determined by the property that it restricts on $\Q_\beta$ to
  $\prec_\beta$ and that all elements from $\Q_{\beta+1}$ are $\prec$-smaller
  than all elements from $\Q_\beta$.
  
  For each $1\le i\le r$, there is a unique $\sigma_i\in\frS_{k_i}$
  with $(i,\sigma_i(j))\prec (i,\sigma_i(j+1))$ for all $1\le j< k_i$:
  we define $\Sigma(\Q)\coloneqq(\sigma_1,\dotsc,\sigma_r)\in\prod_{i=1}^r\frS_{k_i}$.

  The rationale for the previous definition is the following:
  a configuration $z\in\tV_K(\R^{p,1})_{\Sigma}$
  can only witness ray partitions $\Q$ with $\Sigma(\Q)=\Sigma$.
\end{defi}

The following is the main theorem of the section.
\begin{theo}\label{theo:hom}
  Let $p\ge 0$, $q\ge 1$ and $K=(k_1,\dotsc,k_r)$ with $k_i\ge1$.
  \begin{enumerate}
  \item [\upshape\Th{1.}] The integral cohomology
    $H^*\pa{\tV_K(\R^{p,q})}$ is freely generated by
    classes $u_\Q$ for each ray partition, and the cohomological degree of $u_\Q$ is
    \[|u_\Q| = p\cdot \big(r-a(\Q)\big)+(q-1)\cdot \big(|K| - l(\Q)\big).\]
  \item [\upshape\Th{2.}] For $q=1$, the cohomology class $u_\Q$
    is supported on the component $\tV_K(\R^{p,1})_{\Sigma(\Q)}$.
  \end{enumerate}
\end{theo}
\noindent The rest of the section is devoted to the proof of Theorem
\ref{theo:hom}.

\subsection{The weight filtration and the proof of Theorem \ref{theo:hom}}

Throughout this section we fix $K$, $p$ and $q$ as before. We treat
simultaneously the cases $q\ge2$ and $q=1$, putting in parentheses the
differences needed in the case $q=1$.  For $q\ge2$ we abbreviate
$\tV\coloneqq \tV_K(\R^{p,q})$; for $q=1$ we fix
$\Sigma\in\prod_i\frS_{k_i}$ throughout the section and abbreviate
$\tV\coloneqq \tV_K(\R^{p,q})_\Sigma$.

\begin{nota}
  \label{nota:bbP}
  For a positive integer $\Lambda\geq 0$ we denote by $\bbP(\Lambda)$
  the set of all sequences $\lambda=\pa{\lambda_1,\dots,\lambda_l}$
  of integers $\lambda_i\geq 1$, for some $1\leq l\leq \Lambda$,
  satisfying $\lambda_1+\dotsb+\lambda_l=\Lambda$.
  The number $l$ is called the \emph{length} of the sequence.
  
  We have a natural injection
  $\bbP(\Lambda)\hookrightarrow \{0,\dotsc,\Lambda\}^\Lambda$, by adding a
  suitable number of zeroes at the end of each sequence; we consider on
  $\bbP(\Lambda)$ the inherited lexicographic order.

  We denote by $\bbP(K)$ the set $\bbP(\abs{K})$, and by $N$ its cardinality.
\end{nota}

\begin{defi}
  The \emph{weight} of a ray partition $\Q$ is defined as
  \[\omega(\Q)\coloneqq\big(\abs{\Q_1},\dots,\abs{\Q_l}\big)\in\bbP(K).\]
\end{defi}

In the following we state three lemmata and postpone their
proofs to the next subsection.
\begin{lem}
  \label{lem:rayz}
  Let $Z\in \tV$. There is a unique ray partition, called $\Q^Z$,
  which is witnessed by $Z$ and has maximal weight among all ray partitions
  witnessed by $Z$. (If $q=1$, we have moreover that $\Sigma(\Q^Z)=\Sigma$.)
\end{lem}

\begin{defi}
  \label{defi:rayfiltration}
  Given a ray partition $\Q$, we denote by $W_\Q\subset \tV$ the subspace
  containing all points $Z$ with $\Q^Z=\Q$ (see Lemma \ref{lem:rayz}). We
  define a filtration $F_{\bullet}$ on $\tV\cpt$ (see Remark \ref{rem:PLD})
  indexed by the linearly ordered set $\bbP(K)$: for all
  $\lambda\in\bbP(K)$ define the $\lambda$\textsuperscript{th} filtration
  level $F_{\lambda}= F_{\lambda}\tV\cpt$ as the subspace containing $\infty$
  and all $Z\in \tV$ with $\omega(\Q^Z)\geq\lambda$. Note that for
  $\lambda<\lambda'$ in $\bbP(K)$ we have an inclusion
  $F_{\lambda'}\subset F_{\lambda}$.
\end{defi}

\begin{lem}
  \label{lem:filtrationclosed}
  Let $\lambda\in\bbP(K)$.
  Then the inclusion $F_{\lambda}\subseteq \tV\cpt$ is closed.
\end{lem}

\begin{nota}
  \label{nota:filtrationstrata}
  We can switch our indexing set of the filtration $F_{\bullet}$ from Definition
  \ref{defi:rayfiltration} from $\bbP(K)$ to the natural numbers
  $1\leq \nu\leq N$ in the following way: let
  $\chi\colon\set{1,\dots,N}\to\bbP(K)$ be the unique order-\emph{reversing}
  bijection; then for $1\leq \nu\leq N$ we define $F_{\nu}=F_{\chi(\nu)}$.
  Moreover we set $F_0\coloneqq\set{\infty}\subset \tV\cpt$. We obtain an
  \emph{ascending} filtration of $\tV\cpt$ with closed levels (see Lemma
  \ref{lem:filtrationclosed}):
  \[\set{\infty}=F_0 \subseteq F_1\,\subseteq\dotsb\subseteq F_N=\tV\cpt.\]  
  We also denote $F_{-1}\coloneqq\emptyset$, and for $0\leq \nu\leq N$ we denote
  by $\frF_{\nu}$ the $\nu$\textsuperscript{th} filtration stratum of the
  filtration $F_{\bullet}$, i.e.\ the difference
  $\frF_{\nu}=F_{\nu}\setminus F_{\nu-1}$.
\end{nota}

\newpage

\begin{lem}
  \label{lem:stratacontractible}
  The strata satisfy the following properties:
  \begin{enumerate}
  \item [\textup{\bfseries\small 1.}] For each ray partition $\Q$ (with
    $\Sigma(\Q)=\Sigma$), the subspace $W_\Q$ is a contractible open manifold of
    dimension $\abs{K}+p\cdot a(\Q) + (q-1)\cdot l(\Q)$ and a path component
    of the stratum $\frF_{\nu}$, where $1\leq\nu\leq N$ satisfies
    $\chi(\nu)=\omega(\Q)$.
  \item [\textup{\bfseries\small 2.}] All connected components of a stratum
    $\frF_{\nu}$ with $\nu\geq 1$ arise in this way.
  \item [\textup{\bfseries\small 3.}] The closure $\ol{W}_\Q$ of $W_\Q$
    inside $\tV$ is also a smooth, orientable submanifold of dimension
    $\abs{K}+p\cdot a(\Q) + (q-1)\cdot l(\Q)$.
  \end{enumerate}
\end{lem}
We are now ready to prove Theorem \ref{theo:hom}.
\begin{proof}[Proof of Theorem \ref{theo:hom}]
  We consider the Leray spectral sequence associated with the filtered space $\tV\cpt$
  and compute its reduced homology. The $E^1$-page reads
  \[E^1_{\nu,\mu}=H_{\nu+\mu}\pa{F_{\nu},F_{\nu-1}}=\tH_{\nu+\mu}\pa{F_{\nu}/
      F_{\nu-1}}.\]
  By Lemma \ref{lem:stratacontractible}, for all $1\leq \nu\leq N$, $\frF_{\nu}$
  is the disjoint union of the open manifolds $W_{\Q}$ for $\Q$ varying in the
  finite set of ray partitions with $\omega(\Q)=\chi(\nu)$. By Lemma
  \ref{lem:filtrationclosed} we have homeomorphisms
  \[F_\nu/ F_{\nu-1}\cong \frF_\nu\cpt\cong \bigvee_{\omega(\Q)=\chi(\nu)}W_{\Q}\cpt.\]
  Even for $\nu=0$ we have that $F_0=F_0/F_{-1}=\set{\infty}$ is formally homeomorphic
  to the \emph{empty wedge}. By Lemma \ref{lem:stratacontractible}, $W_{\Q}$
  is an open manifold of dimension  $d(\Q)\coloneqq \abs{K}+p\cdot a(\Q) + (q-1)\cdot l(\Q)$
  for all ray partitions; hence we can apply Poincaré–Lefschetz duality and obtain for
  all $\nu,\mu\geq 0$ an isomorphism
  \[E^1_{\nu,\mu}\cong\bigoplus_{\omega(\Q)=\chi(\nu)}
    H_{\nu+\mu}\pa{W_{\Q}\cpt,\infty} \cong\bigoplus_{\omega(\Q)=\chi(\nu)}
    H^{d(\Q)-\nu-\mu}\pa{W_{\Q}}.\]
  Again by Lemma \ref{lem:stratacontractible}, $W_{\Q}$ is contractible for all
  ray partitions $\Q$; hence $H^{d(\Q)-\nu-\mu}\pa{W_{\Q}}$ contributes to the
  first page of the spectral sequence only in the case $\mu+\nu=d(\Q)$. We can
  rewrite, for all $\nu\geq 0$ and considering all degrees $\mu$ at the same time
  \[E^1_{\nu,*}\cong\bigoplus_{\omega(\Q)=\chi(\nu)} H_{d(\Q)}\pa{W_{\Q}\cpt,\infty}.\]
  Since by Lemma \ref{lem:filtrationclosed}, $F_\bullet$ is a closed filtration
  and $W_\caQ$ is a path component of $\frF_\nu$, we can now, for all ray
  partitions $\Q$, replace the relative homology of the pair $(W_{\Q}\cpt,\infty)$
  with the relative homology of the pair $\pa{F_{\nu},F_{\nu}\setminus W_{\Q}}$
  or, by excision, the relative homology of the pair
  $\pa{\ol{W}{}_{\Q}\cpt,\ol{W}{}_{\Q}\cpt\setminus W{}_{\Q}}$. Here, as in
  Lemma \ref{lem:stratacontractible}, we denote by $\ol{W}_{\Q}$ the closure in
  $\tV$ of $W_{\Q}$, and by $\ol{W}{}_{\Q}\cpt$ the one-point compactification of
  $\ol{W}_{\Q}$ (it coincides, for $\nu\ge1$, with the closure of $W_{\Q}$ in
  $\tV\cpt$). We obtain
  \[E^1_{\nu,*}\cong\bigoplus_{\omega(\Q)=\chi(\nu)}
    H_{d(\Q)} \pa{\ol{W}{}_{\Q}\cpt,\ol{W}{}_{\Q}\cpt\setminus W{}_{\Q}}.\]
  Each direct summand in the previous decomposition is isomorphic to $\Z$,
  generated by the fundamental class of the relative manifold
  $\pa{\ol{W}{}_{\Q}\cpt,\ol{W}{}_{\Q}\cpt\setminus W{}_{\Q}}$.

  By Lemma \ref{lem:stratacontractible}, also $\pa{\ol{W}{}_{\Q}\cpt,\infty}$
  is a relative manifold, and its fundamental class projects to that of
  $\pa{\ol{W}{}_{\Q}\cpt,\ol{W}{}_{\Q}\cpt\setminus W{}_{\Q}}$
  under the natural map
  \[H_{d(\Q)}\pa{\ol{W}{}_{\Q}\cpt,\infty} \to
    H_{d(\Q)}\pa{\ol{W}{}_{\Q}\cpt,\ol{W}{}_{\Q}\cpt\setminus W{}_{\Q}}.\]
  The previous analysis shows in particular that for all $\nu\geq 0$ the
  natural map $H_*\pa{F_{\nu},\infty}\to H_*\pa{F_{\nu},F_{\nu-1}}$ is surjective.
  This suffices to prove that the spectral sequence collapses on its first page:
  any element in the first page is represented by a genuine relative cycle of a
  pair $\pa{F_{\nu},\infty}$, so it must survive to the limit. This shows that
  $H_*(\tV^\infty,\infty)$ is freely generated by the fundamental classes of
  the relative submanifolds $(\ol{W}{}^\infty_\Q,\infty)$, so by
  Poincar\'e–Lefschetz duality, $H^*(\tV)$ is generated by their duals,
  which we call $u_\Q$; for each ray partition $\Q$ we finally see
  \begin{align*}
    \abs{u_\Q} &= p\cdot r + q\cdot |K| - d(\Q)\\
               &= p\cdot (r-a(\Q)) + (q-1)\cdot (|K| - l(\Q)).
  \end{align*}
  For $q=1$ and a fixed component $\Sigma\in\prod_i\frS_{k_i}$, the entire argument
  takes place inside $\tV=\tV_K(\R^{p,q})_\Sigma$;
  more precisely, for each ray partition $\Q$ with $\Sigma(\Q)=\Sigma$,
  we have $\ol{W}{}_{\Q}\subset\tV_K(\R^{p,q})_\Sigma$.
  Thus, the second claim of the theorem follows.
\end{proof}

\subsection{Proofs of the three lemmata}

\begin{proof}[Proof of Lemma \ref{lem:rayz}]
  We construct $\Q^Z$ by recursively constructing a sequence
  $(\Q^Z_1,\prec^Z_1),\dotsc,(\Q^Z_\gamma,\prec^Z_\gamma)$ with
  non-empty and disjoint subsets $\Q^Z_1,\dotsc,\Q^Z_\gamma$
  of $T_K$ satisfying the axioms \Th{R1}
  and \Th{R2}, see Figure \ref{fig:algorithm}.
  \begin{itemize}
  \item For ‘$\gamma=1$’, we write $z_1^1=\pa{\zeta_1^1,t_1^1}\in \R^d$, and let
    $\Q^Z_1$ contain all $(i,j)\in T_K$ such that $\smash{z_i^{\smash{j}}}\in
    \R^d$ has the form $\pa{\zeta_1^1,t}$ for some $t\ge t_1^1$; in other words,
    $\Q^Z_1$ contains all $(i,j)\in T_K$ such that $\smash{z_i^{\smash{j}}}$
    lies on the \emph{ray} starting at $z_1^1$ and running in the \emph{positive}
    $t$-direction. The order $\prec^Z_1$ on $\Q^Z_1$ is defined in such a way that
    condition \Th{W2} holds.
  \item For ‘$\gamma-1\to \gamma$’, if $\smash{\Q^Z_1\sqcup\dotsb\sqcup
      \Q^Z_{\gamma-1}\ne T_K}$ have been constructed, let $(i_\gamma,j_\gamma)$
    be the minimal pair in $\smash{T_K\smallsetminus (\Q^Z_1\sqcup\dotsb\sqcup
      \Q^Z_{\gamma-1})}$.
    
    Write $\smash{z_{i_\gamma}^{\smash{j_\gamma}}=
      (\zeta_{i_\gamma}^{\smash{j_\gamma}},t_{i_\gamma}^{\smash{j_\gamma}})}$
    and let $\Q_\gamma^Z$ contain all $(i,j)\in T_K\smallsetminus
    (\Q_1^Z\sqcup\dotsb\sqcup \Q^Z_{\gamma-1})$ such that $z_{i,j}\in\R^d$ has the
    form $\smash{(\zeta_{i_\gamma}^{\smash{j_\gamma}},t)}$ for some $t\ge
    t_{i_\gamma,j_\gamma}$, and define the order $\smash{\prec^Z_\gamma}$ on
    $\Q^Z_\gamma$ in such a way that condition \Th{W2} holds.
  \end{itemize}
  Since $T_K$ is finite, this algorithm terminates, and the resulting
  sequence $\Q^Z\coloneqq(\Q^Z_1,\prec^Z_1),\dotsc,(\Q^Z_{l^Z},\prec^Z_{l^Z})$ is
  a ray partition, which is witnessed by $Z$.

  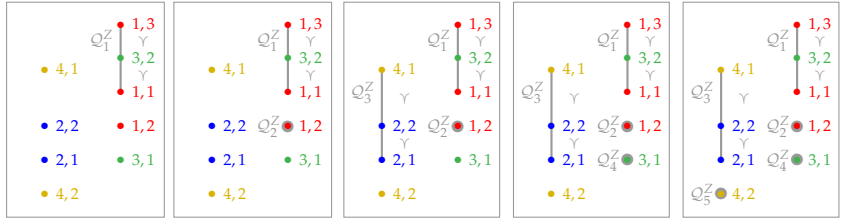
\begin{figure}[h]
    \centering
    \input{fig04}
    \caption{This is how our algorithm proceeds to cover all points.}\label{fig:algorithm}
  \end{figure}

  In the case $q=1$, we have $\Sigma(\Q^Z)=\Sigma$: let $\zeta\in\R^{d-1}$ and
  let $\smash{1\le \beta,\beta'\le l(\Q^Z)}$ be two indices such that for all
  $(i,j)\in \Q_\beta$ and all $\smash{(i',j')\in \Q_{\beta'}}$ we have
  $\smash{\pr_{\zeta}(z_i^{\smash{j}})=\pr_{\zeta}(z_{i'}^{\smash{j'}})=\zeta}$;
  by construction we have
  $\smash{\pr_t(z_{i'}^{\smash{j'}})<\pr_t(z_{i}^{\smash{j}})}$ if and only if
  $\beta'>\beta$ or $\beta'=\beta$ and $(i',j')\prec_\beta (i,j)$. In particular,
  if $i=i'$, we have $\sigma_i(j)<\sigma_i(j')$ if and only if
  $\smash{\pr_t(z_{i'}^{\smash{j'}})<\pr_t(z_i^{\smash{j}})}$, and this holds if
  and only if $(i,j)\prec(i,j')$; this shows that $\Sigma(\Q^Z)=\Sigma$.

  Suppose that $\Q$ is another ray partition witnessed by $Z$, and let
  $l\coloneqq l(\Q)$. We want to show that $\omega(\Q)<\omega(\Q^Z)$ and we do
  this by showing for each $1\le\gamma\le\min(l,l^Z)$ that if
  $\smash{\Q_\beta=\Q_\beta^Z}$ for $\smash{1\le\beta <\gamma}$, then
  $\smash{\Q_\gamma\subseteq \Q^Z_\gamma}$, from which we can deduce inductively
  that since $\Q\ne\Q^Z$ by assumption, there is a $\gamma$ such that
  $\Q_\beta=\Q^Z_\beta$ for $\smash{1\le \beta<\gamma}$ and
  $\smash{\Q_\gamma\subsetneq \Q_\gamma^{Z}}$, so by definition of the
  lexicographic ordering, we get $\smash{\omega(\Q)<\omega(\Q^Z)}$.

  To do so, assume $\Q_\beta=\Q_\beta^Z$ for $\smash{1\le\beta <\gamma}$ and let
  $(i_\gamma,j_\gamma)$ be the minimum of $\smash{T_K\smallsetminus
    (\Q_1\sqcup\dotsb\sqcup \Q_{\gamma-1}) = T_K\setminus
    (\Q_1^Z\sqcup\dotsb\sqcup\Q_{\gamma-1}^Z)}$ as before. By \Th{R1}, the pair
  $(i_\gamma,j_\gamma)$ has to lie inside $\Q_\gamma$, and by \Th{R2} it is the
  minimum with respect to $\prec_\gamma$. By \Th{W1}, all $(i,j)\in \Q_\gamma$
  have to satisfy $\smash{z_{i}^{\smash j}=(\zeta_{i_\gamma}^{\smash{j_\gamma}},t)}$
  for some $t\in\R$, and by \Th{W2}, we additionally require
  $\smash{t\ge t_{i_\gamma}^{\smash{j_\gamma}}}$. Hence
  $\Q_\gamma\subseteq \Q^Z_\gamma$ as  desired.
\end{proof}

\begin{proof}[Proof of Lemma \ref{lem:filtrationclosed}]
  We show that ${\tV}\cpt\setminus F_{\lambda}$ is open. Let $\mathring
  Z\in{\tV}\cpt\setminus F_{\lambda}$, then we have $\omega(\Q^{\smash{\mathring
      Z}})<\lambda$.  Let $\varepsilon>0$ be defined as follows: we consider all
  (Euclidean) distances in $\R^{d-1}$ between any two \emph{distinct} projections
  $\smash{\pr_{\zeta}(\mathring z_{i}^{\smash j})}$ and
  $\smash{\pr_{\zeta}(\mathring z_{i'}^{\smash{j'}})}$, and also all distances in
  $\R$ between any two \emph{distinct} projections $\pr_t(\mathring z_{i}^{\smash
    j})$ and $\pr_t(\mathring z_{i'}^{\smash{j'}})$, for $(i,j)\neq (i',j')$ in
  $T_K$. We obtain a finite set of strictly positive real numbers, and
  $\varepsilon$ is defined as the minimum of all these numbers.

  Let $Z$ be any configuration in $\tV$ such that, for all $(i,j)\in T_K$,
  the distance in $\R^d$ between $\mathring z_{i}^{\smash{j}}$ and $z_{i}^{\smash j}$
  is less than $\tfrac{\varepsilon}{2}$.
  We claim that $\omega(\Q^{Z})\le \omega(\Q^{\smash{\mathring Z}})$:
  the claim implies that for $Z$ in a neighbourhood of $\mathring Z$ in ${\tV}$,
  also $Z\notin F_{\lambda}$. This would conclude the proof, as $\tV$ is open in ${\tV}\cpt$.

  To prove the claim, we use a method similar to the proof of Lemma \ref{lem:rayz},
  namely we show for each $1\le\gamma\le\min(l^{\smash{\mathring Z}},l^{Z})$ that if
  $\smash{\Q_\beta^Z=\Q_\beta^{\smash{\mathring Z}}}$ for all $1\le \beta< \gamma$,
  then $\smash{\Q_\gamma^{Z}\subseteq\Q_\gamma^{\mathring Z}}$, which
  immediately implies $\smash{\omega(\Q^{Z})\le\omega(\Q^{\mathring Z})}$.
  
  The minimum $(i_\gamma,j_\gamma)$ of $\smash{T_K\setminus(\Q^{\mathring
      Z}_1\sqcup\dotsb \sqcup \Q_{\gamma-1}^{\mathring
      Z})=T_K\setminus(\Q^{Z}_1\sqcup\dotsb\sqcup \Q_{\gamma-1}^{Z})}$ has to
  lie in both $\Q_\gamma^{Z}$ and $\Q_\gamma^{\mathring Z}$.  Now for
  $(i,j)\in T_K\setminus (\Q_1^{Z}\sqcup\dotsb\sqcup \Q_{\gamma-1}^{Z})$,
  the following holds: if $\mathring z_{i}^{\smash j}$ does not lie on the ray
  starting at $\mathring z_{i_\gamma}^{\smash {j_\gamma}}$ and running in the
  positive $t$-direction, then either $\pr_{\zeta}(\mathring
  z_{i_\gamma}^{\smash{j_\gamma}})\neq \pr_{\zeta}(\mathring z_{i}^{\smash j})$,
  or the two projections are equal but $\pr_t(\mathring z_{i_\gamma}^{\smash
    {j_\gamma}})>\pr_t(\mathring z_{i}^{\smash j})$. In both cases, by the choice of
  $\varepsilon$, we would also have that $z_{i}^{\smash j}$ does not lie on the
  ray starting at $z_{i_\gamma}^{\smash {j_\gamma}}$ and running in positive
  $t$-direction. This shows in particular that
  $\smash{\Q^{Z}_\gamma\subseteq\Q^{\mathring Z}_\gamma}$ as desired.
\end{proof}

\begin{proof}[Proof of Lemma \ref{lem:stratacontractible}]
  \enlargethispage{\baselineskip}
  Define $H_{\Q}\subseteq{\tV}$ as the subspace of configurations of the form
  $Z=(z_1^1,\dots,z_r^{\smash{k_r}})$ such that the following condition holds: for
  each $1\leq \beta\leq l$ and $(i,j)\prec_{\beta}(i',j')$, we have
  $\pr_{\zeta}(z_i^{\smash j})=\pr_{\zeta}(z_{i'}^{\smash{j'}})$ and
  $\pr_t(z_i^{\smash j})\leq\pr_t(z_{i'}^{\smash{j'}})$.  Then $H_{\Q}$ is a
  closed subspace of ${\tV}$, as it is defined by imposing some equalities and
  some weak inequalities (using `$\leq$') between the coordinates.

  Note, however, that the \emph{same} space $H_\Q$ can be defined, as a subspace
  of $\tV$, by replacing the second condition ‘$\pr_t(z_{i}^{\smash
    j})\leq\pr_t(z_{i'}^{\smash {j'}})$’ with ‘$\pr_t(z_i^{\smash
    j})<\pr_t(z_{i'}^{\smash{j'}})$’. As subspaces of $(\R^d)^{|K|}$, we then
  have the following:
  \begin{itemize}
  \item There is a linear subspace of $(\R^d)^{|K|}$ determined by
    $\pr_1(z_{i}^{\smash j}) =\pr_1(z_{i}^{\smash{j'}})$ for all $1\le i\le r$
    and $1\le j,j'\le k_i$, and inside this linear subspace, $\tV$ is open,
    defined by the strict inequalities
    $z_{i}^{\smash j}\neq z_{i'}^{\smash{j'}}$ for each
    $(i,j)\neq (i',j')\in T_K$;
  \item There is a linear subspace of $(\R^d)^{|K|}$ determined by the linear
    equations
    \begin{itemize}
    \item $\pr_1(z_{i}^{\smash j})=\pr_1(z_{i}^{\smash {j'}})$ for all
      $1\le i\le r$ and $1\le j,j'\le k_i$;
    \item $\pr_{\zeta}(z_i^{\smash j})=\pr_{\zeta}(z_{i'}^{\smash{j'}})$ for all
      $1\le\beta\le l(\Q)$ and $(i,j),(i',j')\in\Q_\beta$;
    \end{itemize}
    inside this linear subspace, $H_\Q$ is open, defined by the strict inequalities
    \begin{itemize}
    \item $z_{i}^{\smash j}\neq z_{i'}^{\smash {j'}}$ for $(i,j)\neq (i',j')\in T_K$;
    \item $\pr_t(z_i^{\smash j})<\pr_t(z_{i'}^{\smash {j'}})$ for all $1\le\beta\le l(\Q)$
      and $(i,j)\prec_{\beta}(i',j')\in\Q_\beta$.
    \end{itemize}
  \end{itemize}
  It follows that $H_\Q$ is an orientable submanifold without boundary of $\tV$,
  which in turn is an orientable submanifold without boundary of $(\R^d)^{|K|}$:
  here we are using the simple observation that the intersection inside a real
  vector space of an open subset and a linear subspace is an orientable
  submanifold without boundary.
  The dimension of $H_\Q$ is computed by noting that, locally, we have the
  following parameters describing a configuration $Z\in H_\Q$.
  \begin{itemize}
  \item For all $1\leq \beta\leq l$, we have a parameter
    $\zeta_\beta=(\zeta^1_\beta,\zeta^2_\beta) \in \R^p\times \R^{q-1}$ which
    corresponds to the (unique) value attained by
    $\pr_{\zeta}(z_{i}^{\smash j})$ for all $(i,j)\in\Q_{\beta}$. However, if
    two rays $\Q_\beta$ and $\Q_{\beta'}$ share a cluster, their further
    projections $\zeta^1_{\smash \beta}$ and $\zeta^1_{\smash{\beta'}}$ have to
    coincide inside $\R^p$. Hence, we get for each equivalence class of rays a
    choice in $\R^p$, and for each ray a choice in $\R^{q-1}$. This yields
    $p\cdot a(\Q) + (q-1)\cdot l(\Q)$ parameters in $\R$.
  \item For each $(i,j)\in T_K$ we have a parameter
    $t_i^{\smash j}=\pr_t(z_i^{\smash j})$ in $\R$.
  \end{itemize}
  We clearly have $W_\Q\subseteq H_\Q$, and ${W}_{\Q}$ can be characterised as
  the subspace of $H_\Q$ containing configurations $Z$ for which the following
  condition holds: for all $1\leq \beta<\beta'\leq l$ and for all
  $(i,j)\in\Q_{\beta}$ and $(i',j')\in\Q_{\beta'}$, either
  $\pr_{\zeta}(z_{i}^{\smash j})\neq \pr_{\zeta}(z_{i'}^{\smash {j'}})$ or
  $\pr_t(z_{i'}^{\smash {j'}}) <\pr_t(z_{i}^{\smash j})$. Thus,
  $W_\Q\subseteq H_\Q$ is an open subspace (it is a finite intersection of open
  subspaces). To see that $W_\Q$ is dense in $H_\Q$, note that one can slightly
  perturb all parameters of any configuration $Z\in H_\Q$ of the forms
  \begin{itemize}
  \item $\zeta_\beta^2\in \R^{q-1}$, for $1\le \beta\le l(\Q)$;
  \item $\zeta_\beta^1\in\R^p$, for $\beta$ ranging in a set of representatives
    of the $a(\Q)$ equivalence classes of rays,
  \end{itemize}
  to ensure that they attain all different values, so that the new perturbed
  con\-figuration $Z'$ lies in $W_\Q$.  Hence, $H_\Q$ is the closure $\ol{W}_\Q$
  inside $\tV$.

  To prove that ${W}_\Q$ is contractible, we choose distinct numbers
  $\mathring t^{\, 1}_1,\dotsc \mathring t_{r}^{\, k_r}\in \R$ such that for
  $1\le \beta,\beta'\le l$ and $(i,j)\in\Q_\beta$ and $(i',j')\in\Q_{\beta'}$,
  we have $\mathring t_{i'}^{\smash {j'}} <\mathring t_{i}^{\smash j}$ if and
  only if $\beta'>\beta$ or $\beta'=\beta$ and $(i',j') \prec_\beta (i,j)$,
  i.e.\ the $\mathring t_{i}^{\smash j}$ are ordered exactly as the stacked
  order on $\Q_l\sqcup\dotsb\sqcup \Q_1$ from Definition \ref{defi:stackedprec}
  prescribes.
  
  We define $\mathring{z}_{i}^{\smash j}\coloneqq (0,\mathring{t}_{i}^{\smash j})\in\R^d$
  for all $(i,j)\in T_K$; note that the configuration
  $\mathring{Z}=(\mathring{z}_{1}^1,\dotsc, \mathring{z}_{r}^{k_r})$ lies in
  ${W}_\Q$. We can connect any configuration $Z\in W_\Q$ to $\mathring{Z}$ by
  \emph{linear interpolation} inside $(\R^d)^{|K|}$, as shown in Figure \ref{fig:linint}: for all $0\leq s\leq 1$
  we consider the configuration $s\cdot Z+(1-s)\cdot \mathring{Z}$, where we set
  \[(s\cdot Z+(1-s)\cdot \mathring{Z})_{i}^{\smash j}\coloneqq s\cdot z_{i}^{\smash j}
    +(1-s)\cdot \mathring{z}_{i}^{\smash j}\in\R^d.\]
  \begin{figure}
    \centering
    \input{fig05}
    \caption{The linear interpolation from $Z$ to the configuration $\mathring{Z}$
      where all points are on the line $(0,0)\times\R$ and the components of the
      ray partition are ordered.}\label{fig:linint}
  \end{figure}
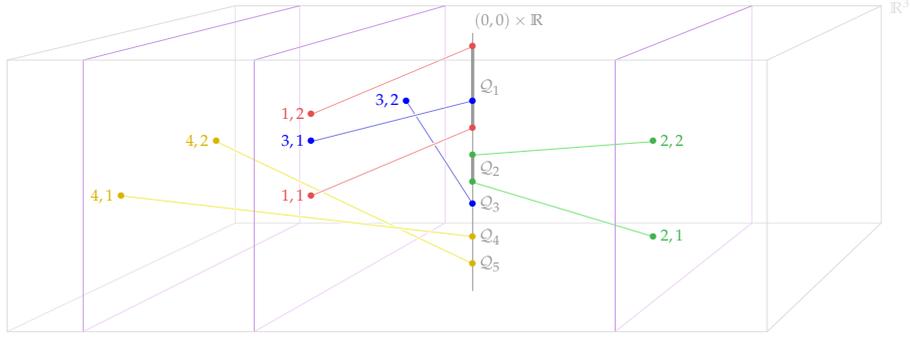
  Since $\pr_{\zeta}$ is a linear map, for all $0\leq s\leq 1$ we have that, for
  fixed $1\leq \beta\leq l$, the map $\pr_{\zeta}$ attains the same value on all
  points of the form
  $s\cdot z_i^{\smash j}+(1-s) \cdot \mathring{z}_{i}^{\smash j}$, for $(i,j)$
  ranging in $\Q_{\beta}$; similarly, for all
  $(i,j)\prec_{\beta}(i',j')\in\Q_{\beta}$ we have an inequality
  $\pr_t((s\cdot Z+(1-s)\cdot \mathring{Z})_{i}^{\smash{j}})<\pr_t((s\cdot
  Z+(1-s)\cdot \mathring{Z} )_{i'}^{\smash {j'}})$.  This means that the linear
  interpolation takes place inside the subspace $\ol{W}_{\Q}$.
  
  With some more effort one can show that the linear interpolation takes place
  inside ${W}_\Q$ by using the characterisation of the points of ${W}_\Q$ inside
  $\ol{W}_\Q$: for all $0<s\leq 1$, $1\leq \beta<\beta'\leq l$,
  $(i,j)\in\Q_{\beta}$, and $(i',j')\in\Q_{\beta'}$, if
  $\pr_{\zeta}(z_{i}^{\smash j})\neq \pr_{\zeta}(z_{i'}^{\smash{j'}})$ then we
  also have the inequality
  \begin{align*}
    \pr_{\zeta}\pa{\big(s\cdot Z+(1-s)\cdot \mathring{Z} \big){}_{i}^{j}}
    \ne\pr_{\zeta}\pa{ \big(s\cdot Z+(1-s)\cdot \mathring{Z} \big){}_{i'}^{j'}},
  \end{align*}
  and if $\pr_t(z_{i'}^{\smash{j'}})<\pr_t(z_{i}^{\smash j})$ then also
  \[\pr_t\pa{\big(s\cdot Z+(1-s)\cdot \mathring{Z} \big){}_{i}^{j}}
    < \pr_t\pa{\big(s\cdot Z+(1-s)\cdot \mathring{Z} \big){}_{i'}^{j'}}, \]
  by the same argument used above. At time $s=1$, we already know that $z$ lies in
  $W_\Q$. Thus, we have exhibited a contraction of $W_\Q$ onto its point $\mathring{Z}$.
  
  In particular we have shown that $W_\Q$ is a connected component of
  $\fF_\nu$. All connected components arise in this way, since every point $Z$
  in $\fF_\nu$ belongs to the subspace $W_{\Q^Z}$, with
  $\chi(\nu)=\omega(\Q^Z)$, see Definition \ref{defi:rayfiltration}.
\end{proof}

\subsection{Growth of Betti numbers and cup product indecomposables}
We will not attempt to give a precise description of $H^*(\tV_K(\R^{p,q}))$
as a ring. The aim of this short subsection is to disprove a natural, yet naïve
conjecture on multiplicative generators of $H^*(\tV_K(\R^{p,q}))$.

\begin{nota}
  For all $1\le i<j\le r$ there is a Fadell–Neuwirth map of the form
  $\pr_{i,j}\colon \tV_K(\R^{p,q})\to\tV_{(k_i,k_{j})}(\R^{p,q})$ which
  forgets all clusters but the $i$\textsuperscript{th} and
  $j$\textsuperscript{th} ones. The map $\pr_{i,j}$ is in general not a
  fibration, though it is a fibration in the quite special case in which $k_l=1$
  for all $l\neq i,j$.
\end{nota}

In the case $k_1,\dotsc,k_r=1$, the space $\tV_K(\R^{p,q})$ is homeomorphic
to the classical ordered configuration space $\tilde{C}_r(\R^d)$, and the maps
$\pr_{i,j}$ reduce to a version of the classical Fadell–Neuwirth fibrations
$\pr_{i,j}\colon \tilde{C}_r(\R^d)\to \tilde{C}_2(\R^d)$. Denote by
$\theta\in H^{d-1}(\tilde{C}_2(\R^d))\cong \Z$ a generator, and let
$\theta_{i,j}\coloneqq\pr_{i,j}^*\theta\in H^{d-1}(\tilde{C}_r(\R^d))$ be the
pulled back cohomology class. It is then a classical result by Arnol'd
\cite{arnold} that the classes $\theta_{i,j}$ generate $H^*(\tilde{C}_r(\R^d))$
as a ring. A natural conjecture would then be the following:
\begin{conj}[Naïve conjecture]
  \label{conj:naive}
  The ring $H^*(\tV_K(\R^{p,q}))$ is generated in arity 2, i.e.\
  by all cohomology classes that can be obtained as a pullback along
  $\smash{\pr_{i,j}}$, for some $1\le i<j\le r$, from a cohomology class in
  $\smash{H^*(\tV_{(k_i,k_{j})}(\R^{p,q}))}$.
\end{conj}
The following example shows that Conjecture \ref{conj:naive} is wrong in general.

\begin{expl}
  Consider the case $r=3$, $k\ge2$, and $p\ge1$ and select the com\-ponent of
  $\smash{\tV_3^k(\R^{p,1})=\tV_{(k,k,k)}(\R^{p,1})}$ corresponding to
  $\Id \coloneqq (\id,\id,\id) \in (\frS_k)^3$.

  Via the stacked total order from
  Definition \ref{defi:stackedprec}, a ray partition $\caQ$ of type
  $K$ with $\Sigma(\caQ)=\on{Id}$ is the same as a \emph{shuffle} of the
  columns of $T_K$, i.e.\ a total order $\prec$ on $T_K$ which preserves the
  ordering of each column: for the inverse construction, given such a shuffle
  $\prec$, we let $\caQ_1$ be the subset of $T_K$ containing $(1,1)$ and all
  $\prec$-larger elements, and $\caQ_\beta$ be the subset of
  $T_K\setminus (\caQ_{1}\sqcup\dotsb\sqcup\caQ_{\beta-1})$ containing the
  $<$-minimal element and all $\prec$-larger ones.
 
  By Theorem \ref{theo:hom}, $\smash{H^*(\tV^k_3(\R^{p,1})_{\Id})}$ is
  concentrated in degrees $0$, $p$ and $2p$, with Betti numbers equal,
  respectively, to the following:
  \begin{itemize}
  \item $1$ in degree $0$. There is indeed a unique ray partition of
    agility $3$, having three rays containing each one cluster.
  \item $3\cdot (\tbinom{2k}k-1)$ in degree $p$. To count ray partitions of
    agility 2, we first choose which of the three clusters forms on its own an
    equivalence class according to Notation \ref{nota:lengthagility}. Without loss
    of generality, we assume to have selected the third cluster to stay on its
    own. The other two clusters can either form a single ray with minimum $(1,1)$,
    for which there are $\binom{2k-1}{k-1}$ possibilities, or they are divided into
    two rays which are equivalent: this second case corresponds to a shuffle of the
    columns $\set{(1,1),\dots,(1,k)}$ and $\set{(2,1),\dots,(2,k)}$, beginning with
    $(2,1)$ and different from the shuffle
    $(2,1)\prec\dotsb\prec(2,k)\prec(1,1)\prec\dotsb\prec(1,k)$; there are
    $\binom{2k-1}k-1$ possibilities for such a shuffle.
  \item $\tbinom{3k}{k}\cdot \tbinom{2k}{k} -3\cdot\tbinom{2k}{k}+2$ in degree
    $2p$. There are precisely $\tbinom{3k}{k}\cdot \tbinom{2k}{k}$
    shuffles of the three columns, and $3\cdot\tbinom{2k}{k}-2$ of these
    shuffles correspond to ray partitions already considered before.
  \end{itemize}
  Similarly, $H^*(\tV^k_2(\R^{p,1})_{\Id})$ is concentrated in degrees $0$ and
  $p$, with Betti numbers equal, respectively, to $1$ and $\binom{2k}k-1$.  The
  projections $\pr_{1,2},\pr_{1,3},\pr_{2,3}$ exhibit an isomorphism between
  $H^p(\tV^k_3(\R^{p,1})_{\Id})$ and
  $H^p(\tV^k_2(\R^{p,1})_{\Id})^{\oplus 3}$.

  If Conjecture \ref{conj:naive} were true, the set of all cup products
  of pairs of classes in $H^p(\tV^k_3(\R^{p,1})_{\Id})$ would suffice to
  generate the entire cohomology group $H^{2p}(\tV^k_3(\R^{p,1})_{\Id})$; in
  particular we would have
  \[\tbinom{3k}{k}\cdot \tbinom{2k}{k} -3\cdot\tbinom{2k}{k}
    +2\le \pa{3\cdot \pa{\tbinom{2k}k-1}}^2.\]
  However, using Stirling's approximation, the left hand side grows as fast as
  $\smash{27^k\cdot \frac{\sqrt{3}}{2\pi k}}$ for $k\to\infty$, whereas the right
  hand side grows as fast as $\smash{9\cdot 16^k\cdot\frac{1}{\pi k}}$.  Hence for
  large $k$ the inequality does not hold; in particular the graded ring
  $H^*(\tV^k_3(\R^{p,1})_{\Id})$ has non-trivial indecomposable elements in
  degree $2p$.
  
  The example generalises for fixed $r\ge4$ to show that the rank of the
  inde-\linebreak composables of $\smash{H^*(\tV^k_r(\R^{p,1})_{\Id})}$ in degree
  $(r-1)\cdot p$ grows as fast as $r^{r\cdot k}$.
\end{expl}
We do not expect that the situation becomes better when considering more than
three clusters, or taking a value of $q$ higher than $1$.

%% file: fig03.tex
\begin{tikzpicture}[xscale=4.5,yscale=2.7]
  \node [thin,grey] at (1.45,1.2) {\tiny $\R^3$};
  \draw [very thin,bgrey] (0,0) -- (.6,.4) -- (1.4,.4);
  \draw [very thin,bgrey]  (.6,.4) -- (.6,1.2);
  \draw [grey,thin] (.3,1.033) -- (.3,0.072);
  \draw [dgrey,very thick] (.3,.25) -- (.3,.5);
  \draw [grey,thin] (.4,1.07) -- (.4,0.139);
  \draw [grey,thin] (.533,1.12) -- (.533,0.233);
  \draw [grey,thin] (.8,1.043) -- (.8,0.082);
  \draw [grey,thin] (1.023,1.14) -- (1.023,0.283);
  \draw [dgrey,very thick] (.8,.5) -- (.8,.7);
  \node [dgrey] at (.399,.496) {$\bullet$};
  \node [dgrey] at (.5315,.8475) {$\bullet$};
  \node [dgrey] at (.5315,.496) {$\bullet$};
  \node [dgrey] at (.8,.8) {$\bullet$};
  \node [dgrey] at (1.0235,.5975) {$\bullet$};
  \draw[very thick,white] (.2,.8) -- (.2,.1);
  \draw [blila,thin] (.2,0) -- (.76,.4) -- (.76,1.2);
  \draw[very thick,white] (.7,.55) -- (.7,0);
  \draw [lila] (.2,0) -- (.2,1) -- (.76,1.2);
  \draw [blila,thin] (.7,0) -- (1.16,.4) -- (1.16,1.2);
  \draw[very thick,white] (.25,1) -- (.95,1);
  \draw[very thick,white] (1,1) -- (1.4,1.2);
  \draw[very thick,white] (1,.55) -- (1,0);
  \draw [lila] (.7,0) -- (.7,1) -- (1.16,1.2);
  \draw [thin,grey] (0,0) rectangle (1,1);
  \draw [thin,grey] (1,0) -- (1.4,.4) -- (1.4,1.2) -- (1,1);
  \draw [thin,grey] (0,1) -- (.6,1.2) -- (1.4,1.2);
  \node [dred] at (.3,.25) {\tiny $\bullet$};
  \node [blue] at (.3,.4) {\tiny $\bullet$};
  \node [dred] at (.3,.5) {\tiny $\bullet$};
  \node [blue] at (.4,.5) {\tiny $\bullet$};
  \node [blue] at (.533,.5) {\tiny $\bullet$};
  \node [dred] at (.533,.85) {\tiny $\bullet$};
  \node [dgreen] at (1.023,.6) {\tiny $\bullet$};
  \node [dgreen] at (.8,.5) {\tiny $\bullet$};
  \node [dgreen] at (.8,.8) {\tiny $\bullet$};
  \node [dgreen] at (.8,.7) {\tiny $\bullet$};
  \node [dred] at (.24,.25) {\tiny $1,2$};
  \node [blue] at (.24,.4) {\tiny $3,3$};
  \node [dred] at (.24,.5) {\tiny $1,3$};
  \node [blue] at (.46,.5) {\tiny $3,2$};
  \node [blue] at (.593,.5) {\tiny $3,1$};
  \node [dred] at (.593,.85) {\tiny $1,1$};
  \node [dgreen] at (1.083,.6) {\tiny $2,1$};
  \node [dgreen] at (.86,.5) {\tiny $2,2$};
  \node [dgreen] at (.86,.8) {\tiny $2,3$};
  \node [dgreen] at (.86,.7) {\tiny $2,4$};
  \node[dgrey] at (.34,0.35) {\tiny $\Q_2$};
  \node[dgrey] at (.48,0.85) {\tiny $\Q_1$};
  \node[dgrey] at (.37,0.58) {\tiny $\Q_7$};
  \node[dgrey] at (.513,0.413) {\tiny $\Q_6$};
  \node[dgrey] at (.74,0.6) {\tiny $\Q_4$};
  \node[dgrey] at (.74,0.8) {\tiny $\Q_5$};
  \node[dgrey] at (.98,0.55) {\tiny $\Q_3$};
\end{tikzpicture}

%% file: fig04.tex
\begin{tikzpicture}[xscale=2.5,yscale=3]
  \draw[thick,dgrey] (.7,.6) -- (.7,.9);
  \node[dgrey] at (.6,.83) {\tiny $\Q^Z_1$};
  \draw[dgrey] (.82,.675) node[rotate=90] {\tiny $\prec$};
  \draw[dgrey] (.82,.825) node[rotate=90] {\tiny $\prec$};
  \draw[dgrey,thin] (.1,0.05) rectangle (.93,1);
  \node[dyellow] at (.3,.7) {\tiny $\bullet$};
  \node[blue] at (.3,.45) {\tiny $\bullet$};
  \node[blue] at (.3,.3) {\tiny $\bullet$};
  \node[dyellow] at (.3,.15) {\tiny $\bullet$};
  \node[red] at (.7,.9) {\tiny $\bullet$};
  \node[dgreen] at (.7,.75) {\tiny $\bullet$};
  \node[red] at (.7,.6) {\tiny $\bullet$};
  \node[red] at (.7,.45) {\tiny $\bullet$};
  \node[dgreen] at (.7,.3) {\tiny $\bullet$};
  \node[dyellow] at (.42,.7) {\tiny $4,1$};
  \node[blue] at (.42,.45) {\tiny $2,2$};
  \node[blue] at (.42,.3) {\tiny $2,1$};
  \node[dyellow] at (.42,.15) {\tiny $4,2$};
  \node[red] at (.82,.9) {\tiny $1,3$};
  \node[dgreen] at (.82,.75) {\tiny $3,2$};
  \node[red] at (.82,.6) {\tiny $1,1$};
  \node[red] at (.82,.45) {\tiny $1,2$};
  \node[dgreen] at (.82,.3) {\tiny $3,1$};
\end{tikzpicture}
\begin{tikzpicture}[xscale=2.5,yscale=3]
  \draw[thick,dgrey] (.7,.6) -- (.7,.9);
  \node[dgrey] at (.6,.83) {\tiny $\Q^Z_1$};
  \draw[dgrey] (.82,.675) node[rotate=90] {\tiny $\prec$};
  \draw[dgrey] (.82,.825) node[rotate=90] {\tiny $\prec$};
  \node[dgrey] at (.7,.448) {$\bullet$};
  \node[dgrey] at (.6,.45) {\tiny $\Q^Z_2$};
  \draw[dgrey,thin] (.1,0.05) rectangle (.93,1);
  \node[dyellow] at (.3,.7) {\tiny $\bullet$};
  \node[blue] at (.3,.45) {\tiny $\bullet$};
  \node[blue] at (.3,.3) {\tiny $\bullet$};
  \node[dyellow] at (.3,.15) {\tiny $\bullet$};
  \node[red] at (.7,.9) {\tiny $\bullet$};
  \node[dgreen] at (.7,.75) {\tiny $\bullet$};
  \node[red] at (.7,.6) {\tiny $\bullet$};
  \node[red] at (.7,.45) {\tiny $\bullet$};
  \node[dgreen] at (.7,.3) {\tiny $\bullet$};
  \node[dyellow] at (.42,.7) {\tiny $4,1$};
  \node[blue] at (.42,.45) {\tiny $2,2$};
  \node[blue] at (.42,.3) {\tiny $2,1$};
  \node[dyellow] at (.42,.15) {\tiny $4,2$};
  \node[red] at (.82,.9) {\tiny $1,3$};
  \node[dgreen] at (.82,.75) {\tiny $3,2$};
  \node[red] at (.82,.6) {\tiny $1,1$};
  \node[red] at (.82,.45) {\tiny $1,2$};
  \node[dgreen] at (.82,.3) {\tiny $3,1$};
\end{tikzpicture}
\begin{tikzpicture}[xscale=2.5,yscale=3]
  \draw[thick,dgrey] (.7,.6) -- (.7,.9);
  \node[dgrey] at (.6,.83) {\tiny $\Q^Z_1$};
  \draw[dgrey] (.82,.675) node[rotate=90] {\tiny $\prec$};
  \draw[dgrey] (.82,.825) node[rotate=90] {\tiny $\prec$};
  \node[dgrey] at (.7,.448) {$\bullet$};
  \node[dgrey] at (.6,.45) {\tiny $\Q^Z_2$};
  \draw[dgrey,thin] (.1,0.05) rectangle (.92,1);
  \draw[thick,dgrey] (.3,.3) -- (.3,.7);
  \draw[dgrey] (.42,.575) node[rotate=90] {\tiny $\prec$};
  \draw[dgrey] (.42,.375) node[rotate=90] {\tiny $\prec$};
  \node[dgrey] at (.2,.6) {\tiny $\Q^Z_3$};
  \node[dyellow] at (.3,.7) {\tiny $\bullet$};
  \node[blue] at (.3,.45) {\tiny $\bullet$};
  \node[blue] at (.3,.3) {\tiny $\bullet$};
  \node[dyellow] at (.3,.15) {\tiny $\bullet$};
  \node[red] at (.7,.9) {\tiny $\bullet$};
  \node[dgreen] at (.7,.75) {\tiny $\bullet$};
  \node[red] at (.7,.6) {\tiny $\bullet$};
  \node[red] at (.7,.45) {\tiny $\bullet$};
  \node[dgreen] at (.7,.3) {\tiny $\bullet$};
  \node[dyellow] at (.42,.7) {\tiny $4,1$};
  \node[blue] at (.42,.45) {\tiny $2,2$};
  \node[blue] at (.42,.3) {\tiny $2,1$};
  \node[dyellow] at (.42,.15) {\tiny $4,2$};
  \node[red] at (.82,.9) {\tiny $1,3$};
  \node[dgreen] at (.82,.75) {\tiny $3,2$};
  \node[red] at (.82,.6) {\tiny $1,1$};
  \node[red] at (.82,.45) {\tiny $1,2$};
  \node[dgreen] at (.82,.3) {\tiny $3,1$};
\end{tikzpicture}
\begin{tikzpicture}[xscale=2.5,yscale=3]
  \draw[thick,dgrey] (.7,.6) -- (.7,.9);
  \node[dgrey] at (.6,.83) {\tiny $\Q^Z_1$};
  \draw[dgrey] (.82,.675) node[rotate=90] {\tiny $\prec$};
  \draw[dgrey] (.82,.825) node[rotate=90] {\tiny $\prec$};
  \node[dgrey] at (.7,.448) {$\bullet$};
  \node[dgrey] at (.7,.3) {$\bullet$};
  \node[dgrey] at (.6,.45) {\tiny $\Q^Z_2$};
  \node[dgrey] at (.6,.3) {\tiny $\Q^Z_4$};
  \draw[dgrey,thin] (.1,0.05) rectangle (.92,1);
  \draw[thick,dgrey] (.3,.3) -- (.3,.7);
  \draw[dgrey] (.42,.575) node[rotate=90] {\tiny $\prec$};
  \draw[dgrey] (.42,.375) node[rotate=90] {\tiny $\prec$};
  \node[dgrey] at (.2,.6) {\tiny $\Q^Z_3$};
  \node[dyellow] at (.3,.7) {\tiny $\bullet$};
  \node[blue] at (.3,.45) {\tiny $\bullet$};
  \node[blue] at (.3,.3) {\tiny $\bullet$};
  \node[dyellow] at (.3,.15) {\tiny $\bullet$};
  \node[red] at (.7,.9) {\tiny $\bullet$};
  \node[dgreen] at (.7,.75) {\tiny $\bullet$};
  \node[red] at (.7,.6) {\tiny $\bullet$};
  \node[red] at (.7,.45) {\tiny $\bullet$};
  \node[dgreen] at (.7,.3) {\tiny $\bullet$};
  \node[dyellow] at (.42,.7) {\tiny $4,1$};
  \node[blue] at (.42,.45) {\tiny $2,2$};
  \node[blue] at (.42,.3) {\tiny $2,1$};
  \node[dyellow] at (.42,.15) {\tiny $4,2$};
  \node[red] at (.82,.9) {\tiny $1,3$};
  \node[dgreen] at (.82,.75) {\tiny $3,2$};
  \node[red] at (.82,.6) {\tiny $1,1$};
  \node[red] at (.82,.45) {\tiny $1,2$};
  \node[dgreen] at (.82,.3) {\tiny $3,1$};
\end{tikzpicture}
\begin{tikzpicture}[xscale=2.5,yscale=3]
  \draw[thick,dgrey] (.7,.6) -- (.7,.9);
  \node[dgrey] at (.6,.83) {\tiny $\Q^Z_1$};
  \draw[dgrey] (.82,.675) node[rotate=90] {\tiny $\prec$};
  \draw[dgrey] (.82,.825) node[rotate=90] {\tiny $\prec$};
  \node[dgrey] at (.7,.448) {$\bullet$};
  \node[dgrey] at (.7,.3) {$\bullet$};
  \node[dgrey] at (.3,.15) {$\bullet$};
  \node[dgrey] at (.6,.45) {\tiny $\Q^Z_2$};
  \node[dgrey] at (.6,.3) {\tiny $\Q^Z_4$};
  \draw[dgrey,thin] (.1,.05) rectangle (.92,1);
  \draw[thick,dgrey] (.3,.3) -- (.3,.7);
  \draw[dgrey] (.42,.575) node[rotate=90] {\tiny $\prec$};
  \draw[dgrey] (.42,.375) node[rotate=90] {\tiny $\prec$};
  \node[dgrey] at (.2,.6) {\tiny $\Q^Z_3$};
  \node[dgrey] at (.2,.15) {\tiny $\Q^Z_5$};
  \node[dyellow] at (.3,.7) {\tiny $\bullet$};
  \node[blue] at (.3,.45) {\tiny $\bullet$};
  \node[blue] at (.3,.3) {\tiny $\bullet$};
  \node[dyellow] at (.3,.15) {\tiny $\bullet$};
  \node[red] at (.7,.9) {\tiny $\bullet$};
  \node[dgreen] at (.7,.75) {\tiny $\bullet$};
  \node[red] at (.7,.6) {\tiny $\bullet$};
  \node[red] at (.7,.45) {\tiny $\bullet$};
  \node[dgreen] at (.7,.3) {\tiny $\bullet$};
  \node[dyellow] at (.42,.7) {\tiny $4,1$};
  \node[blue] at (.42,.45) {\tiny $2,2$};
  \node[blue] at (.42,.3) {\tiny $2,1$};
  \node[dyellow] at (.42,.15) {\tiny $4,2$};
  \node[red] at (.82,.9) {\tiny $1,3$};
  \node[dgreen] at (.82,.75) {\tiny $3,2$};
  \node[red] at (.82,.6) {\tiny $1,1$};
  \node[red] at (.82,.45) {\tiny $1,2$};
  \node[dgreen] at (.82,.3) {\tiny $3,1$};
\end{tikzpicture}

%% file: fig05.tex
\begin{tikzpicture}[xscale=5,yscale=3.6]
  \node [thin,grey] at (2.35,1.2) {\tiny $\R^3$};
  \draw [very thin,bgrey] (0,0) -- (.6,.4) -- (2.3,.4);
  \draw [very thin,bgrey]  (.6,.4) -- (.6,1.2);
  \draw [blila] (.2,0) -- (.77,.4) -- (.77,1.2);
  \draw [blila] (.65,0) -- (1.1525,.4) -- (1.1525,1.2);
  \draw [blila] (1.6,0) -- (1.96,.4) -- (1.96,1.2);
  \draw [thin,grey] (2,0) -- (2.3,.4) -- (2.3,1.2) -- (2,1);
  \draw [thin,grey] (0,1) -- (.6,1.2) -- (2.3,1.2);
  \draw [dgrey,dgrey] (1.225,1.1) -- (1.225,.15);
  \draw [very thick, white] (.55,.7) -- (1.225,.25);
  \draw [thin, byellow] (.55,.7) -- (1.225,.25);
  \draw [very thick,white] (.3,.5) -- (1.225,.35);
  \draw [thin, byellow] (.3,.5) -- (1.225,.35);
  \draw [thin, bgreen] (1.7,.7) -- (1.225,.65);
  \draw [very thick, white] (1.7,.35) -- (1.225,.55);
  \draw [thin, bgreen] (1.7,.35) -- (1.225,.55);
  \draw [very thick,white] (1.05,.85) -- (1.225,.47);
  \draw [thin, bblue] (1.05,.85) -- (1.225,.47);
  \draw [very thick, white] (.8,.5) -- (1.225,.75);
  \draw [thin, bred] (.8,.5) -- (1.225,.75);
  \draw [very thick, white] (.8,.7) -- (1.225,.85);
  \draw [thin, bblue] (.8,.7) -- (1.225,.85);
  \draw [very thick,white] (.8,.8) -- (1.225,1.05);
  \draw [thin, bred] (.8,.8) -- (1.225,1.05);
  \draw [lila] (.65,0) -- (.65,1) -- (1.1525,1.2);
  \draw [lila] (1.6,0) -- (1.6,1) -- (1.96,1.2);
  \draw [lila] (.2,0) -- (.2,1) -- (.77,1.2);
  \draw [thin,grey] (0,0) rectangle (2,1);
  \draw [very thick,dgrey] (1.225,.65) -- (1.225,.55);
  \draw [very thick,dgrey] (1.225,1.05) -- (1.225,.75);
  \node[dgrey] at (1.32,1.145) {\tiny $(0,0)\times \R$};
  \node [dred] at (1.225,1.05) {\tiny $\bullet$};
  \node [blue] at (1.225,.85) {\tiny $\bullet$};
  \node [dred] at (1.225,.75) {\tiny $\bullet$};
  \node [dgreen] at (1.225,.65) {\tiny $\bullet$};
  \node [dgreen] at (1.225,.55) {\tiny $\bullet$};
  \node [blue] at (1.225,.47) {\tiny $\bullet$};
  \node [dyellow] at (1.225,.35) {\tiny $\bullet$};
  \node [dyellow] at (1.225,.25) {\tiny $\bullet$};
  \node [dred] at (.8,.5) {\tiny $\bullet$};
  \node [blue] at (.8,.7) {\tiny $\bullet$};
  \node [dred] at (.8,.8) {\tiny $\bullet$};
  \node [blue] at (1.05,.85) {\tiny $\bullet$};
  \node [dyellow] at (.3,.5) {\tiny $\bullet$};
  \node [dyellow] at (.55,.7) {\tiny $\bullet$};
  \node [dgreen] at (1.7,.7) {\tiny $\bullet$};
  \node [dgreen] at (1.7,.35) {\tiny $\bullet$};
  \node [dred] at (.752,.5) {\tiny $1,1$};
  \node [dyellow] at (.5,.7) {\tiny $4,2$};
  \node [dyellow] at (.25,.5) {\tiny $4,1$};
  \node [blue] at (1,.85) {\tiny $3,2$};
  \node [blue] at (.752,.7) {\tiny $3,1$};
  \node [dred] at (.752,.8) {\tiny $1,2$};
  \node [dgreen] at (1.75,.7) {\tiny $2,2$};
  \node [dgreen] at (1.75,.35) {\tiny $2,1$};
  \node[dgrey] at (1.27,.25) {\tiny $\Q_5$};
  \node[dgrey] at (1.27,.35) {\tiny $\Q_4$};
  \node[dgrey] at (1.27,.47) {\tiny $\Q_3$};
  \node[dgrey] at (1.27,.6) {\tiny $\Q_2$};
  \node[dgrey] at (1.27,.9) {\tiny $\Q_1$};
\end{tikzpicture}

%% file: 04-homological_stability.tex
In this section we will prove homological stability for the unordered
configuration spaces $V^k_r(\R^{p,q})$ of vertical clusters of size $k$ for all
values of $p$ and $q$ except for the one pair where it obviously does not
hold. This extends results by \cite{latifi}, \cite{tran} and \cite{palmerdis}.

\subsection{Setting and results}

We fix throughout the section a cluster size $k\ge 1$ and we will abbreviate
$\tV_r(\R^{p,q})\coloneqq \tV_r^k(\R^{p,q})$ and
$V_r(\R^{p,q})\coloneqq V_r^k(\R^{p,q})$. If $p$ and $q$ are fixed and clear
from the context, we may also just write $\tV_r$ resp. $V_r$.

\begin{constr}
  For each $r\ge 0$, $p\ge 0$, and $q\ge 1$, we have \emph{stabilisation maps}
  \[\stab\colon V^k_r(\R^{p,q})\to V^k_{r+1}(\R^{p,q})\]
  by adding an extra cluster on the ‘far right’ with respect to the first coordinate
  of $\R^{p+q}$, as depicted in Figure \ref{fig:stabilisation}.
\end{constr}

\begin{figure}[h]
  \centering
  \input{fig06}
  \caption{The stabilisation map $\stab\colon V^2_5(\R^{1,1})\to V^2_6(\R^{1,1})$,
    which adds a new cluster on the far right.}\label{fig:stabilisation}
\end{figure}

\begin{prop}
  The induced maps in homology
  \[\stab_*\colon H_m(V_r(\R^{p,q}))\to H_m(V_{r+1}(\R^{p,q})).\]
  are split monic for each $m,r,p\ge 0$ and $q\ge 1$.
\end{prop}
\begin{proof}
  This proof generalises the one from \cite[Lem.\,5.1]{palmerdis} which uses the
  classical idea of a ‘power-set map’: from \cite{dolddec}, we want to use the
  following
  \begin{quote}
    \mbox{\textbf{Lemma~2.}} \emph{Suppose we are given a sequence
      $(0=A_0\stackrel{s_0}{\to} A_1\stackrel{s_1}{\to} \dotsb)$ of abelian groups,
      and assume that there are $\tau_{j,r}:A_r\to A_j$ for $1\le j\le r$ such that
      $\tau_{r,r}=\on{id}$ and $\tau_{j,r}-\tau_{j,r+1}\circ s_r:A_r\to A_j$ lies in
      the image of $s_{j-1}$. Then every $s_r$ is split monic.}
  \end{quote}
  In order to do so, we first note that $V_0=*$, so all spaces
  $V_r=V_r(\R^{p,q})$ are canonically based by $\stab^{r}(*)\in V_r$, and the
  stabilisation maps are basepoint-preserving by definition. For a fixed
  $m\ge 0$, let $A_r\coloneqq \tilde{H}_m(V_r)$, so we have maps
  $s_r\colon A_r\to A_{r+1}$ induced by the stabilisation.

  Now recall for $\ell\ge 0$ the $\ell$-fold symmetric product
  $\on{SP}^\ell V_r\coloneqq (V_r)^\ell/\frS_\ell$ where $\frS_\ell$ acts by
  coordinate permutation. We will denote elements of $\on{SP}^\ell V_r$ as
  formal sums of elements of $V_r$, but we will use a sign
  $\mathring{+}$ resp. $\mathring{\sum}$ in order to distinguish
  the notation for the symmetric product from Notation
  \ref{nota:suggestivesum}. For the binomial coefficient
  $\ell\coloneqq\tbinom{r}{j}$, consider the maps
  \begin{align*}
    \gamma_{j,r}\colon V_r\to \on{SP}^\ell V_j,\quad
    \sum_{i=1}^r[z_i] \mapsto {{\sum^{\circ}_{\substack{S\subseteq\{1,\dotsc,r\}\\\# S=j}}}}
    \sum_{i\in S} [z_i].
  \end{align*}
  A priori, $\gamma_{j,r}$ is not based, but it can be homotoped to a based map
  since $V_r$ is well-based. Then $\gamma_{r,r}=\on{id}$ and we have a homotopy
  \[\gamma_{j,r+1}\circ \on{stab}_r\simeq
    \gamma_{j,r}{~{\mathring{+}}~}
    {\textstyle\on{SP}^{\smash{\tbinom{r}{j-1}}}}\pa{\on{stab}_{j-1}} \circ
    \gamma_{j-1,r}.\]
  of maps $V_r\to\on{SP}^{\binom{r+1}{j}}V_j$. Applying the functor $\pi_m\circ
  \on{SP}^\infty\cong \tilde{H}_m$ and using the ‘flattening’ map
  $\varphi_\ell\colon \on{SP}^\infty\on{SP}^\ell V_r\to \on{SP}^\infty V_r$, we
  obtain the desired system $(\tau_{j,r})$ of morphisms for Dold’s lemma by
  \[
    \begin{tikzcd}[column sep=5em, row sep=1.2em]
      A_r\ar[d,equal]\ar[rr,dashed,"\tau_{j,r}"] && A_j\ar[d,equal]\\
      \pi_m\pa{\on{SP}^\infty V_r}\ar{r}[swap]{\pi_m\pa{\on{SP}^\infty \gamma_{i,r}}} &
      \pi_m\pa{\on{SP}^\infty\on{SP}^\ell V_j}\ar{r}[swap]{\pi_m(\varphi_\ell)} &
      \pi_m\pa{\on{SP}^\infty V_j}.
    \end{tikzcd}\qedhere
  \]
\end{proof}

In contrast, surjectivity of
$\stab_*\colon H_m(V^k_r(\R^{p,q}))\to H_m(V^k_{r+1}(\R^{p,q}))$ holds only in a
certain range. The rest of this section is devoted to the proof of the following
stability theorem:

\begin{theo}\label{theo:stab}
  For all $p\ge 0$ and $q\ge 1$ with $(p,q)\ne (0,1)$, the induced maps
  \begin{align*}
    H_m(V^k_r(\R^{p,q}))\to H_m(V^k_{r+1}(\R^{p,q})).
  \end{align*}
  are isomorphisms for $m\le \tfrac{r}{2}$.
\end{theo}

\noindent Many cases of Theorem \ref{theo:stab} have already been solved:
\begin{itemize}
\item We know that $\pi_0V_r(\R^{0,1})\cong \frS_{kr}/(\frS_k\wr\frS_r)$, so there is no
  stability result to be expected in the case $p=0$ and $q=1$.
\item For $p=0$, we are in the case without any vertical coupling condition.
  This can alternatively be described by embeddings of (disconnected)
  $0$-dimensional manifolds into $\R^q$. For these cases, the theorem was proven
  for $q\ge 3$ in \cite{palmerdis} and for $q= 2$ in \cite{tran}.
\item In \cite{latifi}, the case $p+q\ge 3$ was considered and proven. Actually,
  Latifi writes down the proof only for $p=2$ and $q=1$, but her strategy works
  whenever $p+q\ge 3$.
\end{itemize}
Hence we only have to prove the single remaining case $(p,q)=(1,1)$. However,
since the method is the same, we will provide a proof for arbitrary $(p,1)$ with
$p\ge 1$. Our proof uses different methods than Latifi’s proof.

\subsection{The dexterity filtration}

\begin{nota}
  In the remainder of the section we assume $q=1$. In this case, a vertical
  cluster $[z]=\{z^1,\dotsc,z^k\}\subseteq\R^{p+1}$ is canonically ordered by
  the last coordinate $t^{\smash j}\coloneqq \pr_t(z^{\smash j})\in \R$, and
  $[z]$ is determined by their common projection
  $\zeta\coloneqq\smash{\pr_\zeta}(z)\in\R^p$ and the real numbers
  $t^1,\dotsc,t^k$. Hence, we can write
  \[\{z^1,\dotsc,z^k\} = (\zeta;t^1<\dotsb<t^k).\]
\end{nota}

\begin{defi}\label{defi:dexterity}
  Let $Z\coloneqq \pa{z_1,\dots,z_r}\in \tilde{V}_{r}$ be an ordered
  configuration, where $z_i=(z_{i}^1,\dots,z_{i}^k)$. We define an equivalence
  relation $\sim_Z$ on the set $\set{1,\dots,r}$. First, set $i\sim_Z i'$
  whenever the two following conditions hold:
  \begin{itemize}
  \item $z_i$ and $z_{\smash{i'}}$ are \emph{aligned}, i.e.\ they are contained in the
    same $t$-line, or equivalently  $\pr_{\zeta}(z_i)=\pr_{\zeta}(z_{i'})$ in $\R^{p}$,
    since $q=1$;
  \item $z_i$ and $z_{i'}$ are \emph{entangled}, i.e.\ their convex hulls (contained
    in the vertical line) intersect each other, see Figure \ref{fig:entangle}.
  \end{itemize}
  Let $\sim_Z$ be the equivalence relation generated by the above basic
  relations $\sim_Z$. We define the \emph{dexterity} of $Z$, denoted
  $\delta(Z)$, to be the number $s$ of equivalence classes of $\sim_Z$. Since
  the notion of dexterity is invariant under the permutation action of the group
  $\frS_k\wr\frS_r$, we obtain a notion of dexterity also for unordered
  configurations in $V_r$.
\end{defi}

\begin{figure}[h]
  \centering
  \input{fig07}
  \caption{The leftmost two upper clusters are entangled and hence form an
    equivalence class. Therefore, the dexterity is $5$, while the number
    of clusters is $6$.}\label{fig:entangle}
\end{figure}
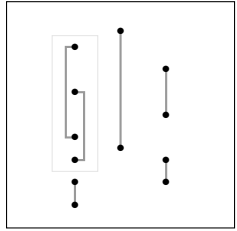

\begin{defi}[Dexterity filtration]
  For $s\ge -1$ we let $F_s V_{r}$ be the subspace of all $[Z]\in V_{r}$ satisfying
  $\delta[Z]\ge r-s$. We have inclusions
  \[\emptyset=F_{-1} V_{r}\subseteq F_0 V_{r}\subseteq\dotsb\subseteq F_{r-1}V_{r}=V_{r}.\]
  We denote by $\frF_s V_r$ the $s$\textsuperscript{th} stratum of the filtration
  \[\frF_s V_r\coloneqq F_sV_r\setminus F_{s-1}V_r\subset V_r.\]
\end{defi}
Note that each filtration level $F_sV_{r}$ is an open subspace of $V_r$, in
particular it is a manifold of the same dimension $p\cdot r+r\cdot k$; the
stratum $\frF_sV_r$ is a closed subset of $F_sV_r$.

Additionally, the stabilisation map $\stab\colon V_{r}\to V_{r+1}$ is
filtration-preserving, i.e.\ it restricts to maps $F_sV_r\to F_sV_{r+1}$ and
even to maps of strata $\frF_sV_r\to \frF_sV_{r+1}$.
\begin{lem}
  The stratum $\frF_sV_r\subset F_sV_r$ is a closed submanifold of codimension
  $s\cdot p$, i.e.\ of dimension $p\cdot (r-s)+r\cdot k$.
\end{lem}
\begin{proof}
  Let $[\mathring{Z}] =\sum_i [\mathring{z}_i]\in \frF_sV_r$ with
  $[\mathring{z}_i]=(\mathring \zeta_i;\mathring t_i^1,\dotsc,\mathring t_i^k)$,
  where $\mathring \zeta_i\in \R^p$ and $\mathring t_i^1<\dotsb< \mathring t_i^k\in \R$.
  A small neighbourhood of $[\mathring{Z}]$ in $V_r$ is described by the following local
  parameters constituting $[Z]=\sum_i(\zeta_i,t_i^1,\dotsc,t_i^{k})$:
  \begin{itemize}
  \item $\zeta_i$, ranging in a neighbourhood of $\mathring{\zeta}_i\in\R^p$,
    for all $1\le i\le r$;
  \item $t_i^{\smash j}$, ranging in a neighbourhood of
    $\mathring t_i^{\smash j}\in\R$, for all $1\le i\le r$ and $1\le j\le k$.
  \end{itemize}
  For $[Z]$ in a neighbourhood of $[\mathring{Z}]$, the condition $[Z]\in\frF_sV_r$ is, up
  to a per\-mutation of clusters of $Z$, equivalent to the equality $\zeta_i=\zeta_{i'}$
  whenever $\smash{i\sim_{\smash{\mathring{Z}}}i'}$.
  
  Let $1\le i_1<\dotsb <i_{r-s}\le r$ be the unique representatives of the $r-s$
  equivalence classes of the relation $\smash{\sim_{\smash{\mathring{Z}}}}$,
  satisfying for each $1\le c\le r-s$ and each
  $\smash{i\sim_{\smash{\mathring{Z}}}i_c}$ the inequality
  $\smash{\mathring t_{i}^1\ge \mathring t_{i_c}^1}$. Thus, a neighbourhood of
  $\smash{[\mathring{Z}]}$ in $\frF_s$ is described by the following local
  parameters:
  \begin{itemize}
  \item $\zeta_{i_{\smash{c}}}$ for all $1\le c\le r-s$;
  \item $t_{i}^{\smash j}$ for all $1\le i\le r$ and $1\le j\le k$.\qedhere
  \end{itemize} 
\end{proof}
\begin{rem}
  \label{rem:normalfrF}
  The argument in the previous proof can be pushed a bit further to describe the
  normal cotangent bundle $N^*(\frF_sV_r,F_sV_r)$ of $\frF_sV_r$ in $F_sV_r$.
  Recall that, for fixed $[\mathring{Z}]\in\frF_sV_r$, the normal cotangent
  space $\smash{N^{\vphantom{\underline{j}}*}_{\smash{[\mathring{Z}]}}(\frF_sV_r,F_sV_r)}$
  is the \emph{subspace} of the cotangent space
  $\smash{T^{\vphantom{\underline{j}}*}_{\smash{[\mathring{Z}]}}F_sV_r}$ of all
  linear functionals that vanish on the vector subspace
  $T_{\smash{[\mathring{Z}]}}\frF_sV_r$. The previous proof shows that
  $N^{\vphantom{\underline{j}}*}_{\smash{[\mathring{Z}]}}(\frF_sV_r,F_sV_r)$ is
  ‘spanned’ by the parameters $\zeta_i$ for
  $i\in\set{1,\dots,r}\setminus\set{i_1,\dots,i_{r-s}}$. By this we formally
  mean the following:
  \begin{itemize}
  \item for each index $i\in\set{1,\dots,r}\setminus\set{i_1,\dots,i_{r-s}}$ we
    consider the list of $p$ linear functionals
    $\mathrm{d}\zeta_i^1,\dots,\mathrm{d}\zeta_i^p$, where
    $\zeta_i^1,\dots,\zeta_i^p\in\R$ are the $p$ coordinates of the parameter
    $\zeta_i$, which takes values in $\R^p$;
  \item a basis for
    $N^{\vphantom{\underline{j}}*}_{\smash{[\mathring{Z}]}}(\frF_sV_r,F_sV_r)$
    is given by the linear functionals $\mathrm{d}\zeta_i^{\smash j}$ where $i$
    ranges in $\set{1,\dots,r}\setminus\set{i_1,\dots,i_{r-s}}$ and
    $1\le j\le p$.
  \end{itemize}
\end{rem}

\subsection{Coloured configuration spaces}

\begin{nota}[Distributions]
  Let $E$ be an index set. A \emph{distribution} is a map $\alpha\colon E\to \N$
  with finite support. We write $\alpha_e\coloneqq \alpha(e)$ and
  $\alpha=\sum_e \alpha_e\cdot e$.

  In particular, for a fixed $e_0\in E$, we denote by $\alpha+e_0$ the
  distribution which coincides with $\alpha$, except for the fact that it
  increases $\alpha_{e_0}$ by $1$.
\end{nota}

\begin{defi}[Coloured labelled configuration spaces]
  Let $E$ be a set and $\alpha\colon E\to \N$ a distribution. Define
  $\lvert \alpha\rvert\coloneqq \sum_{e\in E}\alpha_e$ and
  $\frS({\alpha})\coloneqq \prod_{e\in E}\frS_{\alpha_e}\subseteq
  \frS_{|\alpha|}$. Moreover, let $X\coloneqq (X_e)_{e\in E}$ be a family of
  spaces. Then we define
  \[C_{\alpha}(\R^{p+1};X) \coloneqq \tilde{C}_{|\alpha|}(\R^{p+1})\times_{\frS(\alpha)}
    \prod_{e\in E} X^{\alpha_e}_e.\]
  In case $X_e=*$ for all $e$, we just write
  $C_{\alpha}(\R^{p+1})=\tilde{C}_{\lvert \alpha\rvert}(\R^{p+1})/\frS(\alpha)$.
\end{defi}

Informally, we consider $C_\alpha(\R^{p+1};X)$ as the space of unordered
configurations of $|\alpha|$ unordered points, each equipped with a label in
$\coprod_{e}X_e$, such that for all $e\in E$, there are precisely $\alpha_e$
points carrying a label in $X_e$.

\begin{nota}
  For unordered labelled configurations as before, we will use again the
  suggestive ‘sum notation’: For distinct points
  $y_1,\dotsc,y_{\lvert \alpha\rvert}\in \R^{p+1}$ and labels
  $x_1,\dotsc,x_{\lvert \alpha\rvert}\in X$, we will denote the unordered
  labelled configuration $\{y_1,\dotsc,y_n\}\subseteq \R^{p+1}$ in which the
  point $y_l$ carries the label $x_l$ by
  \[\Theta\coloneqq \sum_{l=1}^{\abs{\alpha}} y_l\otimes x_l\in C_\alpha(\R^{p+1};X).\]
\end{nota}

\begin{defi}
  For $w\ge 1$, an unordered partition of $\{1,\dotsc,w\cdot k\}$ into subsets
  $S_1,\dots,S_w$ of size $k$ is \emph{irreducible} if there is no
  $1\le i\le w-1$ for which the subset $\{1,\dots,i\cdot k\}$ is a union of some
  pieces $S_{b}$ of the partition, see Figure \ref{fig:redpart}. We denote by $\bE_w$ the set of all
  irreducible, unordered partitions of $\{1,\dotsc,w\cdot k\}$.
\end{defi}

\begin{nota}\label{nota:E}
  We denote $\bE\coloneqq \coprod_{w\ge 1}\bE_w$ the union of all $\bE_w$. For
  $e\in \bE_w\subset \bE$, we will write $w(e)\coloneqq w\ge 1$ for the
  \emph{weight} of $e$. Note that there is precisely one partition $e_0\in \bE$
  with $w(e_0)=1$. For all $e=\set{S_1,\dots,S_w}\in \bE_w$, we use the
  convention that $\min(S_b)<\min(S_{b'})$ for $1\le b<b'\le w(e)$.
\end{nota}

\begin{figure}[h]
  \centering
  \input{fig08}
  \caption{We can picture partitions of $\{1,\dotsc,w\cdot k\}$ as in this
    figure. Here we see a reducible partition (left) and an irreducible one
    (right) with $k=2$ and $w=3$.}\label{fig:redpart}
\end{figure}
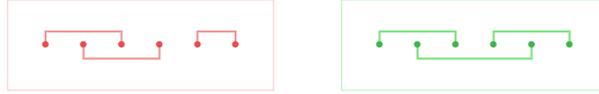

\begin{rem}\label{rem:irredDext}
  The notion of irreducibility is related to the dexterity filtration: given
  an irreducible partition $e=(S_1,\dotsc,S_w)$ with $S_b =
  \set{h_{b}^1<\dotsb<h_b^k}$ and $\zeta_1,\dotsc,\zeta_w\in \R^p$, consider the
  configuration $[Z]\coloneqq \sum_{b=1}^w[z_b^1,\dotsc,z_b^k]\in V_w$ with
  $z_b^{\smash j} = (\zeta_b,h_b^{\smash j})$. Then $[Z]$ has dexterity 1, see
  Definition \ref{defi:dexterity}, if and only if $\zeta_1=\dotsb=\zeta_w$
  holds. This is used in the following, for the special case $\zeta_1=0$.
\end{rem}

\begin{constr}
  For each $e\in \bE$ we denote by $D_e$ the product $(D^p)^{w(e)-1}$, where
  $D^p\subset\R^p$ denotes the standard, Euclidean unit disc; note that $D_e$
  is, topologically, also a disc. We thus obtain a family of discs
  $D\coloneqq (D_e)_{e\in \bE}$, and regard each $D_e$ as a subset of
  $\R^{p\cdot(w(e)-1)}$. For reasons that will become clearer later, we denote
  by $\xi_2,\dots,\xi_{w(e)}$ the $w(e)-1$ parameters of $D_e$, each taking
  values in $D^p$; each parameter $\xi_b$ consists of $p$ coordinates
  $\xi_b^1,\dots,\xi_b^p\in\R$.
  
  For a distribution $\alpha\colon \bE\to \N$, we define the \emph{degree}
  $\deg(\alpha)$ as the pair of non-negative numbers
  \[\deg(\alpha)=(r(\alpha),s(\alpha))\coloneqq
    \pa{\sum_e \alpha_e\cdot w(e) \,,\, \sum_e \alpha_e\cdot \pa{w(e)-1}}.\] We
  let $C_{r,s}\coloneqq C_{r,s}(\R^{p+1};D)$ be the union of all spaces
  $C_\alpha(\R^{p+1};D)$, where $\alpha$ ranges among distributions
  $\alpha\colon \bE\to\N$ with $\deg(\alpha)=(r,s)$.
  
  Each disc $D_e$ contains a centre $0_e\in D_e$, and we denote by
  $0\coloneqq (0_e)_{e\in \bE}$ the family of centres; we obtain an inclusion
  $C_\alpha(\R^{p+1};0)\subseteq C_\alpha(\R^{p+1};D)$ which is a closed
  embedding of a submanifold of codimension $s(\alpha)\cdot p$ for each
  distribution $\alpha\colon \bE\to \N$. We write
  $C_\alpha^*(\R^{p+1};D)\coloneqq C_\alpha(\R^{p+1};D)\setminus
  C_\alpha(\R^{p+1};0)$.

  Moreover, we define $C_{r,s}^0\subseteq C_{r,s}$ to be the union of all
  $C_\alpha(\R^{p+1};0)$ with $\deg(\alpha)=(r,s)$ and define its complement
  $C^*_{r,s}\coloneqq C_{r,s}\setminus C_{r,s}^0$.
\end{constr}

\begin{rem}
  \label{rem:normalCa}
  A generic point in $C_{r,s}$ is of the form
  \[\Theta=\sum_{l=1}^{r-s} y_l\otimes (e_l,\bm{\xi}_l),\]
  where the points $y_1,\dotsc,y_{r-s}\in \R^{p+1}$ are distinct, $e_l\in \bE$,
  and $\bm{\xi}_l\in D_{e_l}$ is expanded as
  $\bm{\xi}_l=(\xi_{l,2},\dotsc,\xi_{\smash{l,w(e_l)}})$ with
  $\xi_{l,b}\in D^p$. Consider a distribution $\alpha$ of degree
  $\deg(\alpha)=(r,s)$, and let
  $\mathring{\Theta}\coloneqq \sum_{l=1}^{r-s} \mathring{y}_l\otimes
  (e_l,0_{e_l})\in C_{\alpha}(\R^{p+1};0)\subseteq C_{r,s}^0$. In order to
  describe a small neighbourhood of $\mathring{\Theta}$ in $C_{r,s}^0$, we can
  use the local parameters $y_l\in\R^{p+1}$, each ranging in a small
  neighbourhood of $\mathring{y}_l$; to describe a small neighbourhood of
  $\smash{\mathring{\Theta}}$ in $C_{r,s}$ we can additionally use the
  parameters $\bm{\xi}_l=(\xi_{l,2},\dotsc,\xi_{l,w(e)})\in D_{e_l}$, each
  ranging in a small neighbourhood of $0_{e_l}$. It follows that the
  $N^*_{\mathring{\Theta}}(C_{r,s}^0,C_{r,s})$ is ‘spanned by the parameters
  $\bm{\xi}_l$ for $1\le l\le r-s$’. By this we mean the following:
  \begin{itemize}
  \item for each $1\le l\le r-s$ we consider the list of $p\cdot(w(e_l)-1)$
    linear functionals $\mathrm{d}\xi_{l,b}^\tau$, for $2\le b\le w(e_l)$ and
    $1\le \tau\le p$. Here $\xi_{l,b}^\tau$, for $2\le b\le w(e_l)$ and
    $1\le \tau\le p$, are the $p\cdot (w(e_l)-1)$ coordinates of the parameter
    $\bm{\xi}_l$, which takes values in $D_{e_l}\subset \R^{p\cdot (w(e_l)-1)}$;
  \item a basis for $N^*_{\mathring{\Theta}}(C_{r,s}^0,C_{r,s})$ is given by all
    linear functionals $\mathrm{d}\xi_{l,b}^\tau$ indexed by $1\le l\le r-s$,
    $2\le b\le w(e_l)$, and $1\le \tau\le p$.
  \end{itemize}
\end{rem}

There are stabilisation maps $C_{r,s}\to C_{r+1,s}$ given by placing a new point
with label in $D_{e_0}=*$ on the ‘right’ with respect to the first
coordinate of $\R^{p+1}$. The stabilisation increases the parameter $r$ by $1$,
but leaves $s$ constant. Moreover, $C^0_{r,s}$ is sent to $C^0_{r+1,s}$ under
the stabilisation map, and $C^*_{r,s}$ is sent to $C^*_{r+1,s}$

\subsection{The insertion map}

We will now connect the filtration pairs $(F_sV_r,F_{s-1}V_r)$ to the pairs
$(C_{r,s},C^*_{r,s})$ of labelled configurations via an ‘insertion map’.

\begin{constr}
  For each $1\le s\le r$, we have a map of pairs
  \[\varphi_{r,s}\colon (C_{r,s},C^*_{r,s}) \to (F_sV_r,F_{s-1}V_r),\]
  which pictorially does the following, see Figure \ref{fig:insert}: given a
  labelled configuration in $C_{r,s}$, we draw pairwise disjoint cylinders
  around each point in $\R^{p+1}$ and place inside each of them a small
  ‘standard configuration’ which corresponds to the given indecomposable
  partition, and is ‘perturbed’ by the $w(e)-1$ disc parameters from the label,
  where in each cylinder, one cluster stays in the centre.

  By Remark \ref{rem:irredDext}, the dexterity of the resulting vertical
  configuration is $r-s$ if and only if \emph{all} clusters inside each cylinder
  stay in the centre, i.e.\ if all disc parameters are $0$. Thus, if the labelled
  configuration to start with lies in $C_{r,s}^*$, then the dexterity is at least
  $r-(s-1)$, whence we land in the filtration component $F_{s-1}V_r$. Formally,
  the map $\phi_{r,s}$ is constructed as follows:
  \begin{itemize}
  \item For each $w\ge 1$, each subset $S\subseteq \{1,\dotsc,w\cdot k\}$ of
    cardinality $k$, and each $\xi\in D^p$, we define the unordered ‘standard
    cluster’
    \[T_S(\xi) \coloneqq
      \pa{\xi,\pa{-1+\tfrac{2}{k\cdot w+1}\cdot h}_{h\in S}}.\]
    Pictorially, $T_S(\xi)$ is the unordered vertical cluster of $k$ points
    which projects to $\xi\in D^p$ and whose $t$-coordinate takes the values
    corresponding to $S$, among all values arising from a uniform distribution of
    $w\cdot k$ points in the interior of the interval $[-1;1]$.
  \item For a partition $e\in \bE$ write $S_1,\dotsc,S_{w(e)}\subseteq\{1,\dotsc,w\cdot k\}$
    for the partition components; since $e$ is an unordered partition, we assume
    without loss of generality that $\min(S_i)<\min(S_{i+1})$. For all
    $\xi_2,\dotsc,\xi_{w(e)}\in D^p$, we set $\xi_{1}\coloneqq 0\in D^p$ and define
    \[T_e(\xi_2,\dotsc,\xi_{w(e)}) \coloneqq
      \sum_{b=1}^{w(e)} T_{S_b}(\xi_{b})\in V_{w(e)}.\]
    Note that the ‘lowest cluster’ (the one attaining the lowest value of the
    $t$-coordinate) is ‘in the middle’ (projects to the centre of $D^p$). Note also
    that $S_b$ and $S_{\smash{b'}}$ are disjoint for $b\neq b'$, hence the clusters
    $T_{S_b}(\xi_b)$ and $\smash{T_{S_{\smash{b'}}}(\xi_{\smash{b'}})}$ are also
    disjoint, and the sum defining $\smash{T_e(\xi_2,\dotsc,\xi_{w(e)})}$ is
    well-defined.
  \item Consider on $\R^p\times \R$ the ‘product distance’
    \[d((\zeta,t),(\zeta',t')) \coloneqq \max(d(\zeta,\zeta'),d(t,t'))\]
    with respect to the Euclidean distances on $\R^p$ resp. $\R$. This means
    that for a radius $\rho>0$, the closed $\rho$-ball around $(\zeta,t)$ is given
    by the cylinder $B_{\rho}(\zeta,t)=(\zeta+\rho\cdot D^p)\times [t-\rho;t+\rho]$.
  \end{itemize}
  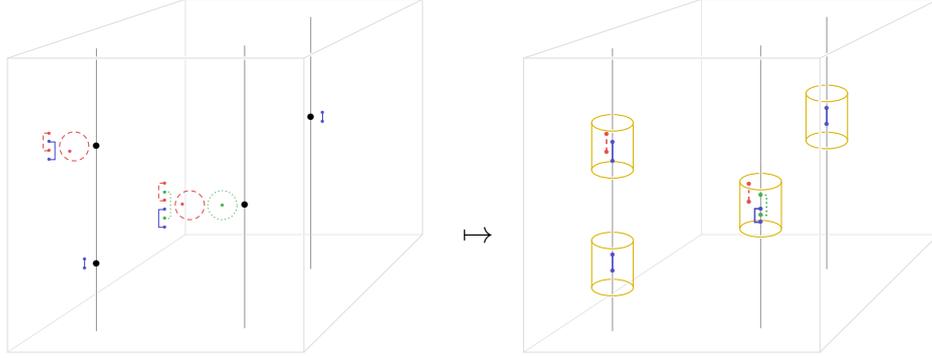
\begin{figure}[t]
    \centering
    \input{fig09}
    \caption{An instance of the insertion map
      $\varphi_{7,3}\colon (C_{7,3},C^*_{7,3})\to (F_3V_7,F_2V_7)$.
      The result even lies in the deeper filtration component $F_1V_7$.}\label{fig:insert}
  \end{figure}
  \noindent Now we have everything together to define the desired insertion map:
  given an element $\Theta=\sum_{l} y_l\otimes (e_l,\bm{\xi}_l)$ in $C_{r,s}$,
  we define
  \[\rho=\rho(\Theta) \coloneqq
    \begin{cases}
      \tfrac15\cdot \min_{l\ne l'}d(y_l,y_{l'}) & \text{for $r-s\ge 2$},\\
      1                                         & \text{for $r-s=1$.}
    \end{cases}\]
  and accordingly, the map
  \[\varphi_{r,s}\colon C_{r,s}\to F_sV_r,\quad
    \Theta=\sum_{l=1}^{r-s} y_l\otimes (e_l,\bm{\xi}_l)\mapsto
    \sum_{l=1}^{r-s} \pa{y_l + \rho(\Theta)\cdot T_{e_l}(\bm{\xi}_l)}.\]
  First of all, note that the signs ‘$+$’ and ‘$\cdot$’ in the expression ‘$y_l
  + \rho(\Theta)\cdot T_{e_l}(\bm{\xi}_l)$’ denote a translation and a dilation in
  $\R^{p+1}$; the sum sign always describes an unordered collection (of points, or
  of vertical clusters). Note that the second sum is well-defined as the
  configurations $y_l+\rho(\Theta)\cdot T_{e_l}(\bm{\xi}_l)$ lie inside the
  cylinders $B_{\rho}(y_l)$ and are hence disjoint. Finally, we have
  $\delta(\varphi_{r,s}(\Theta))\ge r-s$, so the image of $\varphi_{r,s}$ is
  actually contained in the filtration level $F_sV_r\subset V_r$.

  As indicated before, Remark \ref{rem:irredDext} ensures that
  $\varphi_{r,s}(\Theta)$ lies in the stratum $\frF_sV_r$ if and only if
  $\bm{\xi}_l=0_{e_l}$ for all $1\le l\le r-s$. In particular $\varphi_{r,s}$
  restricts to maps $C^*_{r,s}\to F_{s-1}V_r$ and $C^0_{r,s}\to \frF_sV_r$.
\end{constr}

\begin{rem}
  \begin{enumerate}
  \item [\Th{1.}] In fancier language, and up to homotopy, we
    constructed the following: Let $D\coloneqq \coprod_e D_e$, and define a map
    $\varphi\colon D\to V:=\coprod_{r\ge 0}V_r$, restricting for all $e\in \bE$
    to a map $D_e\to V_{w(e)}$ similar to the map $T_e$ above. Then we use that
    $V$ is an $\scC_{p+1}$-algebra and consider the adjoint map of
    $\scC_{p+1}$-algebras, with source the free $\scC_{p+1}$-algebra on $D$:
    \[\varphi\colon C(\R^{p+1};D)\simeq F^{\scC_{p+1}}(D)\to V.\]
    The left hand side decomposes as a disjoint union of the spaces $C_{r,s}$ as
    before, while the right hand side decomposes into the spaces $V_r$, which
    are filtered by the spaces $F_sV_r$, and the map $\varphi$ is compatible
    with this decomposition and filtration. We spelled out the insertion map
    directly for three reasons: firstly, we need the explicit choice of
    $\rho$ later in the proof; secondly, this adjoint description gives a
    definition of $\varphi$ which is sensible only up to homotopy, and in
    particular we could not immediately make sense of a statement like:
    \emph{$\varphi$ is compatible with the filtrations and their strata},
    and thirdly, we need each $\phi_{r,s}$ to be smooth in the proof of the
    upcoming Proposition \ref{prop:varphiiso}.
  \item[\Th{2.}] The insertion maps respect the stabilisation maps
    on both sides up to homotopy: the stabilisation maps in the following
    diagram of maps of pairs can be chosen so that the diagram commutes on the
    nose
    \[
      \begin{tikzcd}
        (C_{r,s},C^*_{r,s}) \ar[r,"\varphi_{r,s}"]\ar[d,swap,"\stab"] & (F_sV_r,F_{s-1}V_r)\ar[d,"\stab"]\\
        (C_{r+1,s},C^*_{r+1,s})\ar[r,swap,"\varphi_{r+1,s}"]          & (F_sV_{r+1},F_{s-1}V_{r+1})
      \end{tikzcd}
    \]
  \item[\Th{3.}] In order to ensure that $\varphi_{r,s}$ is
    well-defined, it would have been enough to choose $\rho$ slightly smaller
    than $\frac{1}{2}\cdot\min_{l\ne l'}d(y_l,y_{l'})$. However, we wanted to
    ensure that $\varphi_{r,s}\colon C_{r,s}\to F_sV_r$ is even an embedding, so
    we need to ensure that $\varphi_{r,s}(\Theta)$ still ‘knows’ which clusters
    come from the same label.
  \end{enumerate}
\end{rem}

\begin{prop}
  \label{prop:varphiiso}
  For each $1\le s\le r$, the map
  $\varphi_{r,s}\colon(C_{r,s}, C^*_{r,s})\to (F_sV_r,F_{s-1}V_r) $
  induces an isomorphism in relative homology.
\end{prop}
The proof of Proposition \ref{prop:varphiiso} relies on the following lemma.
\begin{lem}\label{lem:varphiiso}
  For $1\le s\le r$, the map
  $\varphi_{r,s}\colon C^0_{r,s} \to \frF_sV_r$ is a homotopy equivalence.
\end{lem}
We first finish the proof of Proposition \ref{prop:varphiiso} assuming Lemma
\ref{lem:varphiiso}, and then prove Lemma \ref{lem:varphiiso}.
\begin{proof}[Proof of Proposition \ref{prop:varphiiso}]
  Since $\frF_sV_r\subseteq F_sV_r$ is an embedded, closed submanifold, the
  relative homology groups of $(F_sV_r,F_{s-1}V_r)$ can be computed using excision
  and the Thom isomorphism for the normal bundle of $\frF_sV_r$ in $F_sV_r$,
  \[H_*(F_sV_r,F_{s-1}V_r)
    =H_*(F_sV_r, F_sV_r\setminus \frF_s V_r)
    \cong H_{*-s\cdot p}(\frF_sV_r;\caO_{r,s}),\]
  for a suitable choice of local coefficients $\caO_{r,s}$ on $\frF_sV_r$.
  Similarly, the relative homology groups of $(C_{r,s}, C^*_{r,s})$ are given by
  \[H_*(C_{r,s}, C^*_{r,s})=H_*(C_{r,s}, C_{r,s}\setminus C^{\smash 0}_{\smash{r,s}})
    \cong H_{*-s\cdot p}(C^0_{r,s};\caO'_{r,s}),\]
  for a suitable choice of local coefficients $\caO'_{r,s}$ on $C^0_{r,s}$.

  Recall Remark \ref{rem:normalfrF} and Remark \ref{rem:normalCa}: the key
  observation is that the insertion map $\varphi_{r,s}$ induces a map of normal
  bundles $N(C^0_{r,s},C_{r,s})\to N(\frF_sV_r,F_sV_r)$, which is an isomorphism
  on fibres. To see this, fix $\mathring{\Theta}\in C^0_{r,s}$ as in Remark
  \ref{rem:normalCa}, denote $[\mathring{Z}]=\varphi_{r,s}(\mathring{\Theta})$,
  and use the notation from Remark \ref{rem:normalfrF}. By construction
  $\phi_{r,s}$ is a smooth map, and our aim is to check that
  \[\phi_{r,s}^*\colon N^*_{[\mathring Z]}(\frF_sV_r,F_sV_r)
    \to N^*_{\mathring{\Theta}}(C^0_{r,s},C_{r,s})\]
  is an isomorphism of vector spaces. This follows directly from the observation
  that $\phi_{r,s}^*$ sends the linear functional $\mathrm{d}\zeta_i^{\smash j}$
  to the linear functional $\smash{\frac{1}{\rho(\mathring{\Theta})}\cdot
    \mathrm{d}\xi_{l,b}^\tau}$, whenever the cluster $[\mathring{z}_i]$ of the
  collection $[\mathring{Z}]$ is obtained from the labelled point
  $\mathring{y}_l\otimes \bm{\xi}_l$, using the $b$\textsuperscript{th} component
  $S_b$ of the partition $e_l$. Hence the local coefficient system $\caO'_{r,s}$
  coincides with $\varphi_{r,s}^*\caO_{r,s}$, and so,
  \[
    \begin{tikzcd}[column sep=4em]
      H_*(C_{r,s}, C^*_{r,s})\ar[r,"\cong","\text{Thom}"'] \ar[d,"(\varphi_{r,s})_*"'] &
      H_{*-s\cdot p}(C^0_{r,s};\caO'_{r,s})\ar[d,"(\varphi_{r,s})_*"]\\
      H_*(F_sV_r, F_{s-1}V_r) \ar[r,"\cong","\text{Thom}"'] & H_{*-s\cdot p}(\frF_sV_r;\caO_{r,s}).
    \end{tikzcd}
  \]
  commutes. By Lemma \ref{lem:varphiiso}, the right vertical map is an
  isomorphism; hence also the left vertical map is an isomorphism.
\end{proof}

\begin{proof}[Proof of Lemma \ref{lem:varphiiso}]
  The insertion map $\varphi_{r,s}\colon C^0_{r,s}\hookrightarrow \frF_sV_r$ is
  a closed injection, and we define a deformation retraction of $\frF_sV_r$ onto $C^0_{r,s}$:
  for each \mbox{$[Z]\in \frF_sV_r$} there is, up to permutation, a unique sequence
  $e_1,\dotsc,e_{r-s}\in \bE$ such that
  \[[Z]=\sum_{l=1}^{r-s}\sum_{S\in e_l}\pa{\zeta_l,(u_{\smash l}^{\smash h})_{h\in S}},\]
  where we assume $u_{l}^{\smash h}<u_{l}^{h+1}$ for all $1\le h< w(e_l)\cdot k$. We define
  \[u_l\coloneqq \tfrac{1}{k\cdot w(e_l)}\cdot (u_{l}^1+\dotsb+u_{l}^{\smash{k\cdot w(e_l)}})\]
  Moreover, let $y_l \coloneqq (\zeta_{l},u_l)\in \R^{p+1}$
  and $\smash{\rho\coloneqq \rho(y_1,\dotsc,y_{r-s})}$.
  The homotopy $H\colon \frF_sV_r\times[0;1]\to \frF_sV_r$ is given by
  linear interpolation of the local parameters:
  we move $u_{l}^h$ to $\smash{u_l+\rho\cdot\big({-1} + \frac{2}{k\cdot w(e_l)+1}\cdot h\big)}$,
  and keep the values $\zeta_l$ fixed. For each $1\le l\le r-s$, the
  interpolation takes place in the convex hull of the points
  $(\zeta_l,u_{l}^h)\in\R^{p+1}$; these are $r-s$ disjoint vertical segments in
  $\R^{p+1}$, and at each time the original vertical order of the points lying on
  each of these segments is preserved. Therefore no collision between distinct
  points of the vertical configuration occurs.

  We note that $H([Z],1)$ is in the subspace $\varphi_{r,s}(C^0_{r,s})$, and by
  construction, $H$ fixes pointwise at all times the subspace
  $\varphi_{r,s}(C^0_{r,s})$.
\end{proof}

\subsection{The stability proof}

\begin{lem}\label{lem:colstab}
  The stabilisation map $C_{r,s}\to C_{r+1,s}$ induces isomorphisms
  \[H_m\pa{C_{r,s},C^*_{r,s}} \to H_m(C_{r+1,s},C_{r+1,s}^*)\]
  in homology for $m\le \tfrac{r}{2}$ and all $1\le s\le r$.
\end{lem}
\begin{proof} 
  For each distribution $\alpha$ of degree $(r,s)$, we have a (not always orientable)
  disc bundle $C_\alpha(\R^{p+1};D)\to C_\alpha(\R^{p+1})$ of dimension $p\cdot s$ 
  and structure group $\frS(\alpha)$, which gives us a Thom isomorphism
  \[H_{m}\pa{C_\alpha(\R^{p+1};D),C^*_\alpha(\R^{p+1};D)}
    \cong H_{m-p\cdot s}\pa{C_{\alpha}(\R^{p+1});\hspace*{1px}\pr_\alpha^*\caO_\alpha}
    =:M_{m,\alpha},\]
  where $\pr_\alpha\colon \pi_1(C_{\alpha}(\R^{p+1}))\to \frS(\alpha)$ is the projection
  and $\caO_\alpha$ is of the form
  \[\caO_\alpha\colon \frS(\alpha)\to \{\pm 1\},
    \quad(\sigma_e)_{e\in \bE}\mapsto \prod_e \sg(\sigma_e)^{p\cdot (w(e)-1)}.\]
  In particular, we have a natural isomorphism
  $\pr^*_\alpha\caO_{\alpha}\cong \stab^*\pr^*_{\alpha+e_0}\caO_{\alpha+e_0}$
  for the canonical stabilisation $C_\alpha(\R^{p+1}) \to C_{\alpha+e_0}(\R^{p+1})$,
  which gives us induced stabilisation morphisms $M_{m,\alpha}\to M_{m,\alpha+e_0}$.
  Now we have a commutative square
  \[
    \begin{tikzcd}
      H_m(C_{r,s},C^*_{r,s})\ar[d,swap]\ar[r,"\cong"] &
      \displaystyle\bigoplus_{\deg(\alpha)=(r,s)}M_{m,\alpha}\ar[d]\\
      H_m(C_{r+1,s},C^*_{r+1,s})\ar[r,swap,"\cong"] &
      \displaystyle\bigoplus_{\deg(\alpha)=(r,s)}M_{m,\alpha+e_0}\oplus
      \bigoplus_{\substack{\deg(\alpha)=(r+1,s)\\\alpha_{e_0}=0}} M_{m,\alpha}
    \end{tikzcd}
  \]
  where the left vertical arrow is the desired stabilising map and the right
  side is a sum of maps $M_{m,\alpha}\to M_{m,\alpha+e_0}$. Therefore, we can
  prove the statement by showing that for $m\le \frac{r}{2}$, we have
  $M_{m,\alpha}=0$ for each distribution $\alpha$ of degree $(r+1,s)$ with
  $\alpha_{e_0}=0$, and secondly, that the stabilising map $M_{m,\alpha}\to
  M_{m,\alpha+e_0}$ is an isomorphism for each distribution $\alpha$ of degree
  $(r,s)$. For the first part, we use $w(e)\ge 2$ for all $e\in \bE$ with
  $\alpha_e\neq 0$ to obtain
  \[p\cdot s\ge p\cdot \sum_e \alpha_e\cdot \pa{w(e)-\tfrac 12\cdot  w(e)}
    =p\cdot \tfrac{r+1}2\ge \tfrac{r+1}2,\]
  so $m-p\cdot s<0$, whence $M_{m,\alpha}=0$. For the second part, have to check
  that $H_{m-p\cdot s}(C_\alpha(\R^{p+1});\pr_{\alpha}^*\caO_\alpha)\to
  H_{m-p\cdot s}(C_{\alpha+e_0}(\R^{p+1});\pr_{\alpha+e_0}^*\caO_{\alpha+e_0})$
  is an isomorphism for $m\le \frac{r}{2}$. To do so, we first observe that
  $\tfrac{r}{2} \ge \tfrac{1}{2}\cdot\alpha_{e_0} + \sum_{e\ne e_0}\alpha_e$ since
  $w(e)\ge 2$ for $e\ne e_0$; and as $p\ge 1$, we obtain
  \begin{align*}
    m-p\cdot s &\le \tfrac{r}{2}-p\cdot \sum_e \alpha_e\cdot (w(e)-1)
                 \le -\tfrac{r}{2} + \sum_{e}\alpha_e
                 \le \tfrac 12\cdot \alpha_{e_0}.
  \end{align*}
  Now we want to use a technique from \cite{palmertwist}, so we adapt his
  notation by writing $\lambda\coloneqq (\alpha_e)_{e\ne e_0}$, so
  $\abs{\lambda} = \sum_{e\ne e_0}\alpha_e$ as well as $\lambda[n]$ for the
  distribution with $\lambda[n](e_0)=n-\abs{\lambda}$ and
  $\lambda[n](e)=\alpha_e$ for $e\ne e_0$, so we have $\lambda[r-s]=\alpha$ and
  $\lambda[r-s+1] = \alpha+e_0$. This notation has the advantage that
  $n=\sum_e\lambda[n](e)$. We have a stabilisation map
  $C_{\lambda[n]}(\R^{p+1})\to C_{\lambda[n+1]}(\R^{p+1})$ by placing an
  additional point with label $e_0$, which for $n=r-s$ is our map from before.

  Now we construct a signed version of \cite[Ex.\,4.6]{palmertwist}: let
  $\mathbf{PInj}$ be the category whose objects are non-negative integers and
  whose morphisms $n\to n'$ are partially defined injections $\eta\colon
  \{1,\dotsc,n\}\dashrightarrow \{1,\dotsc, n'\}$. We define a functor
  $\caP_\lambda\colon \mathbf{PInj}\to \mathbf{Ab}$ to the category of abelian
  groups as follows: we set
  \[\caP_\lambda(n) \coloneqq \Z\llangle{
      (P_e)_{e\ne e_0};\,\text{
        $P_e\subseteq \set{1,\dotsc,n}$,
        $P_e\cap P_{e'}=\emptyset$, and
        $\# P_e=\lambda_e$%
      }
    },\]
  and for each partially-defined injection $\eta\colon n\to n'$
  and $P\coloneqq (P_e)_{e\ne e_0}$, we define
  \[\eta_*(P) \coloneqq
    \begin{cases}
      \prod_{e\ne e_0}\sg(\eta|_{P_e})^{p\cdot (w(e)-1)}\cdot (\eta(P_e))_{e\ne e_0}
      & \text{if $\eta$ is defined on $\bigcup_e P_e$},\\
      0 & \text{else,}
    \end{cases}\]
  where in the first case, the restriction $\eta|_{P_e}\colon P_e \to \eta(P_e)$
  can canonically be identified with a permutation in $\frS_{\alpha_e}$ since $P_e$
  and $\eta(P_e)$ are totally or\-dered as subsets of $\{1<\dotsb<n\}$ resp.\
  $\{1<\dotsb<n'\}$. By the same in\-ductive argument as in
  \cite[Lem.\,4.7]{palmertwist}, $\caP_\lambda$ is a polynomial coefficient system
  with \mbox{$\deg(\caP_\lambda)=\abs{\lambda}=r-s-\alpha_{e_0}$} and since
  $\Z[\frS_n]\otimes_{\frS(\lambda[n])} \caO_{\lambda[n]}\cong \caP_\lambda(n)$ as
  $\frS_n$-representations, we have natural isomorphisms
  \[
    \begin{tikzcd}
      H_{m'}\pa{C_{\lambda[n]}(\R^{p+1});\hspace*{1px}\pr^*_{\lambda[n]}\caO_{\lambda[n]}}\ar[d,swap,"\cong"]\ar[r]& H_{m'}\pa{C_{\lambda[n+1]}(\R^{p+1});\hspace*{1px}\pr^*_{\lambda[n+1]}\caO_{\lambda[n+1]}}\ar[d,"\cong"]\\
      H_{m'}\pa{C_n(\R^{p+1});\hspace*{1px}\pr^*_n\caP_\lambda(n)}\ar[r] & H_{m'}\pa{C_{n+1}(\R^{p+1});\hspace*{1px}\pr^*_{n+1}\caP_\lambda(n+1)},
    \end{tikzcd}
  \]
  where $\pr_n\colon \pi_1\pa{C_n(\R^{p+1})}\to \frS_n$ is the projection. Since
  $p+1\ge 2$, the bottom map is an isomorphism for $m'\le \frac{1}{2}\cdot
  (n-r+s+\alpha_{e_0})$ by \cite[Thm.\,A]{palmertwist}. For us, $n=r-s$, we get
  the isomorphism for $m'\le \frac{1}{2}\cdot \alpha_{e_0}$ as desired.
\end{proof}

Now we have all tools to prove Theorem \ref{theo:stab}.

\begin{proof}[Proof of Theorem \ref{theo:stab}]
  This is now a standard argument: let $E(r)$ denote the Leray homology spectral
  sequence associated with the filtered space $V_r(\R^{p,1})$; the fil\-tration-preserving
  stabilisation $V_r(\R^{p,1})\to V_{r+1}(\R^{p,1})$ induces a morphism $f\colon
  E(r)\to E(r+1)$ of spectral sequences, and on the first page we have exactly the
  morphisms
  \[f^1_{s,t}\colon E(r)^1_{s,t}
    =H_{s+t}(F_sV_r,F_{s-1}V_r)\to E(r+1)^1_{s,t}
    =H_{s+t}(F_sV_{r+1},F_{s-1}V_{r+1}).\]
  By Proposition \ref{prop:varphiiso} and Lemma \ref{lem:colstab}, we know that
  $f^1_{s,t}$ is an isomorphism for $s+t\le \tfrac{r}{2}$, so by a standard
  comparison argument between spectral sequences \cite{Zeeman} we obtain
  that $H_m(V_r)\to H_m(V_{r+1})$ is an isomorphism for $m\le \frac r2$.
\end{proof}

\begin{outl}
  \begin{enumerate}
  \item [\Th{1.}]As already remarked, the space
    $V^k(\R^{p,q})=\coprod_{r\ge0}V_r^k(\R^{p,q})$ is a $\scC_{p+q}$-algebra with
    $\N$ as monoid of path components, hence the stable homology
    $H_*(V^k_{\infty}(\R^{p,q})):=\mathrm{colim}_{r\to \infty}H_*(V_r^k(\R^{p,q}))$
    agrees with the homology of (a component of) some $\Omega^{p+q}$-space. The
    second author \cite{kranhold2} provides a geometric model for the $p$-fold
    delooping of $V^k(\R^{p,q})$, and in the case $q=1$ even for the $(p+1)$-fold
    delooping. We still lack a geometric description of the $(p+q)$-fold delooping
    of $V^k(\R^{p,q})$ for $q\ge2$, even in the seemingly innocent case $p=0$.
  \item [\Th{2.}] We believe that the Leray spectral sequence for
    the filtration $F_\bullet V_r$ collapses on its first page and that the
    extension problem is trivial. This would then imply that, using the notation
    from the proof of Lemma \ref{lem:colstab},
    \[H_m(V_r(\R^{p,1}))\cong\bigoplus_{s=0}^{r-1} H_m(C_{r,s},C^*_{r,s})
      \cong \bigoplus_{s=0}^{r-1}\bigoplus_{\deg(\alpha)=(r,s)}M_{m,\alpha}.\]
    Our motivation is again the description of the stable homology
    \[H_m(V_{\infty}(\R^{p,1}))\cong \bigoplus_{s=0}^{\infty}\bigoplus_{\alpha}M_{m,\alpha},\]
    given in \cite{kranhold2}, where the last direct sum is extended over all
    distributions $\alpha\colon \bE\to\N$ with $\alpha_{e_0}=0$ and $s(\alpha)=s$. 
  \item [\Th{3.}] The strategy of proof of Theorem \ref{theo:stab}
    generalises to the following case: Let $K=(k_1,\dots,k_r)$, and for $k\ge1$
    let $r(k)\ge0$ be the number of indices $1\le i\le r$ with $k_i=k$; define a
    stabilisation map $V_K(\R^{p,1})\to V_{(K,k)}(\R^{p,1})$ by inserting a
    new vertical cluster; then the induced map in homology
    \[\stab_*\colon H_m(V_K(\R^{p,1}))\to H_m(V_{(K,k)}(\R^{p,1}))\]
    is an isomorphism for $m\le \frac{r(k)}{2}$. We leave to the reader the
    details of the generalisation of the proof.
  \end{enumerate}
\end{outl}

%% file: fig06.tex
\begin{tikzpicture}[scale=3]
  \draw[thin,dgrey] (0,0) rectangle (1,1);
  \draw[dgrey,thick] (.7,.7) -- (.7,.5);
  \draw[dgrey,thick] (.7,.3) -- (.7,.2);
  \draw[dgrey,thick] (.4,.87) -- (.4,.35);
  \draw[dgrey,thick] (.3,.8) -- (.26,.8) -- (.26,.4) -- (.3,.4);
  \draw[dgrey,thick] (.3,.6) -- (.34,.6) -- (.34,.3) -- (.3,.3);
  \node at (.3,.3) {\tiny $\bullet$};
  \node at (.3,.4) {\tiny $\bullet$};
  \node at (.3,.6) {\tiny $\bullet$};
  \node at (.3,.8) {\tiny $\bullet$};
  \node at (.7,.5) {\tiny $\bullet$};
  \node at (.7,.7) {\tiny $\bullet$};
  \node at (.7,.2) {\tiny $\bullet$};
  \node at (.7,.3) {\tiny $\bullet$};
  \node at (.4,.87) {\tiny $\bullet$};
  \node at (.4,.35) {\tiny $\bullet$};
\end{tikzpicture}\quad\raisebox{38px}{$\mapsto$}\quad
\begin{tikzpicture}[scale=3]
  \draw[thin,bgrey] (1,0) -- (1,1);
  \draw[thin,dgrey] (0,0) rectangle (1.5,1);
  \draw[dgrey,thick] (.7,.7) -- (.7,.5);
  \draw[dgrey,thick] (.7,.3) -- (.7,.2);
  \draw[dgrey,thick] (.4,.87) -- (.4,.35);
  \draw[dgrey,thick] (.3,.8) -- (.26,.8) -- (.26,.4) -- (.3,.4);
  \draw[dgrey,thick] (.3,.6) -- (.34,.6) -- (.34,.3) -- (.3,.3);
  \draw[bblue,thick] (1.25,.4) -- (1.25,.6);
  \node at (.3,.3) {\tiny $\bullet$};
  \node at (.3,.4) {\tiny $\bullet$};
  \node at (.3,.6) {\tiny $\bullet$};
  \node at (.3,.8) {\tiny $\bullet$};
  \node at (.7,.5) {\tiny $\bullet$};
  \node at (.7,.7) {\tiny $\bullet$};
  \node at (.7,.2) {\tiny $\bullet$};
  \node at (.7,.3) {\tiny $\bullet$};
  \node at (.4,.87) {\tiny $\bullet$};
  \node at (.4,.35) {\tiny $\bullet$};
  \node[dblue] at (1.25,.4) {\tiny $\bullet$};
  \node[dblue] at (1.25,.6) {\tiny $\bullet$};
\end{tikzpicture}

%% file: fig07.tex
\begin{tikzpicture}[scale=3]
  \draw[thin] (0,0) rectangle (1,1);
  \draw[dgrey,thick] (.7,.7) -- (.7,.5);
  \draw[dgrey,thick] (.7,.3) -- (.7,.2);
  \draw[dgrey,thick] (.5,.87) -- (.5,.35);
  \draw[dgrey,thick] (.3,.8) -- (.26,.8) -- (.26,.4) -- (.3,.4);
  \draw[dgrey,thick] (.3,.1) -- (.3,.2);
  \draw[dgrey,thick] (.3,.6) -- (.34,.6) -- (.34,.3) -- (.3,.3);
  \node at (.3,.3) {\tiny $\bullet$};
  \node at (.3,.4) {\tiny $\bullet$};
  \node at (.3,.1) {\tiny $\bullet$};
  \node at (.3,.2) {\tiny $\bullet$};
  \node at (.3,.6) {\tiny $\bullet$};
  \node at (.3,.8) {\tiny $\bullet$};
  \node at (.7,.5) {\tiny $\bullet$};
  \node at (.7,.7) {\tiny $\bullet$};
  \node at (.7,.2) {\tiny $\bullet$};
  \node at (.7,.3) {\tiny $\bullet$};
  \node at (.5,.87) {\tiny $\bullet$};
  \node at (.5,.35) {\tiny $\bullet$};
  \draw[thin,bgrey] (.2,.85) rectangle (.4,.25);
\end{tikzpicture}

%% file: fig08.tex
\begin{tikzpicture}[xscale=.5,yscale=.6]
  \draw[thick,bred] (0,0) -- (0,.3) -- (2,.3) -- (2,0);
  \draw[thick,bred] (1,0) -- (1,-.3) -- (3,-.3) -- (3,0);
  \draw[thick,bred] (4,0) -- (4,.3) -- (5,.3) -- (5,0);
  \node[dred] at (0,0) {\tiny $\bullet$};
  \node[dred] at (1,0) {\tiny $\bullet$};
  \node[dred] at (2,0) {\tiny $\bullet$};
  \node[dred] at (3,0) {\tiny $\bullet$};
  \node[dred] at (4,0) {\tiny $\bullet$};
  \node[dred] at (5,0) {\tiny $\bullet$};
  \draw[thin,bred!50] (-1,-1) rectangle (6,1);
\end{tikzpicture}
\qquad
\begin{tikzpicture}[xscale=.5,yscale=.6]
  \draw[thick,bgreen] (0,0) -- (0,.3) -- (2,.3) -- (2,0);
  \draw[thick,bgreen] (1,0) -- (1,-.3) -- (4,-.3) -- (4,0);
  \draw[thick,bgreen] (3,0) -- (3,.3) -- (5,.3) -- (5,0);
  \node[dgreen] at (0,0) {\tiny $\bullet$};
  \node[dgreen] at (1,0) {\tiny $\bullet$};
  \node[dgreen] at (2,0) {\tiny $\bullet$};
  \node[dgreen] at (3,0) {\tiny $\bullet$};
  \node[dgreen] at (4,0) {\tiny $\bullet$};
  \node[dgreen] at (5,0) {\tiny $\bullet$};
  \draw[thin,bgreen!50] (-1,-1) rectangle (6,1);
\end{tikzpicture}

%% file: fig09.tex
\tikzstyle{very densely dashed}=[dash pattern=on 1.8pt off 1.3pt]
\begin{tikzpicture}[xscale=3.9,yscale=3.9]
  \draw [very thin,bgrey] (0,0) -- (.6,.4) -- (1.4,.4);
  \draw [very thin,bgrey]  (.6,.4) -- (.6,1.2);
  \draw[very thick,white] (.3,.8) -- (.3,.1);
  \draw[very thick,white] (.8,1) -- (.8,.3);
  \draw[very thick,white] (1.023,.55) -- (1.023,0);
  \draw [dgrey,thin] (.3,1.033) -- (.3,0.072);
  \draw [dgrey,thin] (.8,1.043) -- (.8,0.082);
  \draw [dgrey,thin] (1.023,1.14) -- (1.023,0.283);
  \draw[very thick,white] (1,1) -- (1.4,1.2);
  \draw[very thick,white] (1,.55) -- (1,0);
  \draw[very thick,white] (0.1,1) -- (.9,1);
  \draw [grey] (0,0) rectangle (1,1);
  \draw [grey] (1,0) -- (1.4,.4) -- (1.4,1.2) -- (1,1);
  \draw [grey] (0,1) -- (.6,1.2) -- (1.4,1.2);
  \node at (.3,.3) {\tiny $\bullet$};
  \node at (.3,.7) {\tiny $\bullet$};
  \node at (1.023,.8) {\tiny $\bullet$};
  \node at (.8,.5) {\tiny $\bullet$};
  \draw[dred,very densely dashed] (.14,.745) -- (.12,.745) -- (.12,.685) -- (.14,.685);
  \draw[dblue] (.14,.715) -- (.16,.715) -- (.16,.655) -- (.14,.655);
  \node[dred] at (.14,.743) {$\cdot$};
  \node[dblue] at (.14,.713) {$\cdot$};
  \node[dred] at (.14,.683) {$\cdot$};
  \node[dblue] at (.14,.653) {$\cdot$};
  \draw[dred,very densely dashed] (.225,.7) circle (1.4pt);
  \node[dred] at (.21,.681) {$\cdot$};
  \draw[dblue] (.26,.315) -- (.26,.285);
  \node[dblue] at (.26,.313) {$\cdot$};
  \node[dblue] at (.26,.283) {$\cdot$};
  \draw[dblue] (1.063,.815) -- (1.063,.785);
  \node[dblue] at (1.063,.813) {$\cdot$};
  \node[dblue] at (1.063,.783) {$\cdot$};
  \draw[dgreen,densely dotted] (.725,.5) circle (1.4pt);
  \draw[dred,very densely dashed] (.615,.5) circle (1.4pt);
  \node[dred] at (.59,.5) {$\cdot$};
  \node[dgreen] at (.725,.498) {$\cdot$};
  \draw[dred,very densely dashed] (.53,.575) -- (.51,.575) -- (.51,.515) -- (.53,.515);
  \draw[dgreen,densely dotted] (.53,.545) -- (.55,.545) -- (.55,.455) -- (.53,.455);
  \draw[dblue] (.53,.485) -- (.51,.485) -- (.51,.425) -- (.53,.425);
  \node[dred] at (.53,.573) {$\cdot$};
  \node[dgreen] at (.53,.543) {$\cdot$};
  \node[dred] at (.53,.513) {$\cdot$};
  \node[dblue] at (.53,.483) {$\cdot$};
  \node[dgreen] at (.53,.453) {$\cdot$};
  \node[dblue] at (.53,.423) {$\cdot$};
\end{tikzpicture}
\quad\raisebox{42px}{$\mapsto$}\quad
\begin{tikzpicture}[xscale=3.9,yscale=3.9]
  \draw [very thin,bgrey] (0,0) -- (.6,.4) -- (1.4,.4);
  \draw [very thin,bgrey]  (.6,.4) -- (.6,1.2);
  \draw[very thick,white] (.3,.8) -- (.3,.1);
  \draw[very thick,white] (.8,1) -- (.8,.3);
  \draw[very thick,white] (1.023,.55) -- (1.023,0);
  \draw [dgrey,thin] (.3,1.033) -- (.3,0.072);
  \draw [dgrey,thin] (.8,1.043) -- (.8,0.082);
  \draw [dgrey,thin] (1.023,1.14) -- (1.023,0.283);
  \draw[very thick,white] (1,1) -- (1.4,1.2);
  \draw[very thick,white] (1,.55) -- (1,0);
  \draw[very thick,white] (0.1,1) -- (.9,1);
  \draw [grey] (1,0) -- (1.4,.4) -- (1.4,1.2) -- (1,1);
  \draw [grey] (0,1) -- (.6,1.2) -- (1.4,1.2);
  \draw[very thick,white] (.3,.38) ellipse (2pt and .8pt);
  \draw[very thick,white] (.3,.22) ellipse (2pt and .8pt);
  \draw[thin,dyellow] (.3,.38) ellipse (2pt and .8pt);
  \draw[thin,dyellow] (.3,.22) ellipse (2pt and .8pt);
  \draw[very thick, white] (.3,.38) -- (.3,.45);
  \draw[dgrey,thin] (.3,.38) -- (.3,.45);
  \draw[very thick, white] (.3,.22) -- (.3,.28);
  \draw[dgrey,thin] (.3,.22) -- (.3,.28);
  \draw[dyellow] (.23,.22) -- (.23,.38);
  \draw[dyellow] (.37,.22) -- (.37,.38);
  \draw[very thick,white] (1.023,.88) ellipse (2pt and .8pt);
  \draw[very thick,white] (1.023,.72) ellipse (2pt and .8pt);
  \draw[thin,dyellow] (1.023,.88) ellipse (2pt and .8pt);
  \draw[thin,dyellow] (1.023,.72) ellipse (2pt and .8pt);
  \draw[very thick, white] (1.023,.88) -- (1.023,.95);
  \draw[dgrey,thin] (1.023,.88) -- (1.023,.95);
  \draw[very thick, white] (1.023,.72) -- (1.023,.78);
  \draw[dgrey,thin] (1.023,.72) -- (1.023,.78);
  \draw[dyellow] (.953,.72) -- (.953,.88);
  \draw[dyellow] (1.093,.72) -- (1.093,.88);
  \node[dred] at (.76,.554) {$\cdot$};
  \draw[dred,very densely dashed,semithick] (.76,.557) -- (.76,.511);
  \draw[very thick,white] (.3,.78) ellipse (2pt and .8pt);
  \draw[very thick,white] (.3,.62) ellipse (2pt and .8pt);
  \draw[thin,dyellow] (.3,.78) ellipse (2pt and .8pt);
  \draw[thin,dyellow] (.3,.62) ellipse (2pt and .8pt);
  \draw[very thick, white] (.3,.78) -- (.3,.85);
  \draw[dgrey,thin] (.3,.78) -- (.3,.85);
  \draw[very thick, white] (.3,.62) -- (.3,.68);
  \draw[dgrey,thin] (.3,.62) -- (.3,.68);
  \draw[dyellow] (.23,.62) -- (.23,.78);
  \draw[dyellow] (.37,.62) -- (.37,.78);
  \draw[very thick,white] (.8,.58) ellipse (2pt and .8pt);
  \draw[very thick,white] (.8,.42) ellipse (2pt and .8pt);
  \draw[thin,dyellow] (.8,.58) ellipse (2pt and .8pt);
  \draw[thin,dyellow] (.8,.42) ellipse (2pt and .8pt);
  \draw[very thick, white] (.8,.58) -- (.8,.65);
  \draw[dgrey,thin] (.8,.58) -- (.8,.65);
  \draw[very thick, white] (.8,.42) -- (.8,.48);
  \draw[dgrey,thin] (.8,.42) -- (.8,.48);
  \draw[dyellow] (.73,.42) -- (.73,.58);
  \draw[dyellow] (.87,.42) -- (.87,.58);
  \node[dred] at (.28,.676) {\Large $\cdot$};
  \node[dblue] at (.3,.71) {\Large$\cdot$};
  \node[dred] at (.28,.739) {\Large$\cdot$};
  \draw[dblue,semithick] (.3,.3267) -- (.3,.2733);
  \node[dblue] at (.3,.3267) {\Large $\cdot$};
  \node[dblue] at (.3,.2733) {\Large $\cdot$};
  \node[dblue] at (.3,.647) {\Large $\cdot$};
  \draw[dred,very densely dashed,semithick] (.28,.68) -- (.28,.744);
  \draw[dblue,semithick] (.3,.716) -- (.3,.652);
  \node[dgreen] at (.8,.531) {\Large $\cdot$};
  \node[dred] at (.76,.508) {\Large $\cdot$};
  \node[dred] at (.76,.57) {\Large $\cdot$};
  \node[dblue] at (.8,.485) {\Large $\cdot$};
  \node[dgreen] at (.8,.463) {\Large $\cdot$};
  \node[dblue] at (.8,.439) {\Large $\cdot$};
  \draw[dgreen,densely dotted,semithick] (.8,.534) -- (.82,.534) -- (.82,.466) -- (.8,.466);
  \draw[dblue,semithick] (.8,.442) -- (.78,.442) -- (.78,.488) -- (.8,.488);
  \draw[dblue,semithick] (1.023,.8267) -- (1.023,.7733);
  \node[dblue] at (1.023,.8267) {\Large $\cdot$};
  \node[dblue] at (1.023,.7733) {\Large $\cdot$};
  \draw[very thick, white] (1,.5) -- (1,.95);
  \draw [grey] (0,0) rectangle (1,1);
\end{tikzpicture}